\documentclass[12pt,a4paper]{article}%
\usepackage{amsmath}
\usepackage{graphicx}
\usepackage{epsfig}%
\usepackage{amsfonts}%
\usepackage{amssymb}
\usepackage{tikz} \usetikzlibrary{3d,positioning,shapes}
\usetikzlibrary{scopes,shapes.geometric,shadows}
\usepackage{pgfplots} \usetikzlibrary{plotmarks}
\usetikzlibrary{calc,backgrounds,fit}
\usepackage{color}

\usepackage[normalem]{ulem}

\newcommand{\tb}[1]{\textbf{#1}}
\renewcommand{\it}[1]{$\mathit{#1}$}

\newcommand{\pc}[1]{\left[ #1 \right]}
\newcommand{\pl}[1]{\left\{ #1 \right \}}

\newtheorem{theorem}{Theorem}[section]

\newtheorem{definition}[theorem]{Definition}

\newtheorem{lemma}[theorem]{Lemma}

\newtheorem{remark}[theorem]{Remark}
\newtheorem{summary}[theorem]{Summary}

\newenvironment{proof}[1][Proof]{\noindent \emph{#1.} }{\hfill \ 
\rule{0.5em}{0.5em}}

\makeatletter\@addtoreset{equation}{section}\makeatother
\makeatletter\@addtoreset{figure}{section}\makeatother
\makeatletter\@addtoreset{table}{section}\makeatother
\begin{document}
\title{Tucker tensor method for fast grid-based summation of 
long-range potentials on 3D lattices with defects
}

\author{Venera Khoromskaia\thanks{Max-Planck-Institute for
        Mathematics in the Sciences, Inselstr.~22-26, D-04103 Leipzig,
        Germany ({\tt vekh@mis.mpg.de}).} \and
        Boris N. Khoromskij\thanks{Max-Planck-Institute for
        Mathematics in the Sciences, Inselstr.~22-26, D-04103 Leipzig,
        Germany ({\tt bokh@mis.mpg.de}).}
        }
 
\date{}

\maketitle

\begin{abstract}
In this paper, we present a method for fast summation
of long-range potentials on 3D lattices with multiple defects and having non-rectangular geometries,
based on rank-structured tensor representations.
This is a significant generalization of our recent technique for the grid-based
summation of electrostatic potentials on the rectangular $L\times L \times L$ lattices
by using the canonical tensor decompositions and yielding the $O(L)$ computational complexity
instead of $O(L^3)$ by traditional approaches.
The resulting  lattice sum is calculated as a Tucker 
or canonical representation whose directional vectors are assembled by the 1D summation of the 
generating vectors for the shifted reference tensor, once precomputed on large 
$N\times N \times N$ representation grid in a 3D bounding box.  The tensor numerical treatment of defects
is performed in an algebraic way by simple summation of tensors in the canonical or Tucker formats. 
To diminish the considerable increase in the tensor rank of the resulting potential sum
the $\varepsilon$-rank reduction procedure is applied based on the 
generalized reduced higher-order SVD scheme.
For the reduced higher-order SVD approximation to a sum of canonical/Tucker tensors,
we prove the stable error bounds in the relative norm  in terms of discarded 
singular values of the side matrices.
The required storage scales linearly in the 1D grid-size, 
$O(N)$, while the numerical cost is estimated by $O(N L)$. 
The approach applies to a general class of kernel functions  including those for
the Newton, Slater, Yukawa, Lennard-Jones, and dipole-dipole interactions.
Numerical tests confirm the efficiency of the presented tensor summation method:
we demonstrate that a sum of millions of Newton kernels on a 3D lattice with 
defects/impurities can be computed in seconds in Matlab implementation.
The tensor approach is advantageous in further functional calculus
with the lattice potential sums represented on a 3D grid, like integration or differentiation,
using tensor arithmetics of 1D complexity.
\end{abstract}



\section{Introduction}\label{sec:introduct}

Efficient methods for computation of a sum of classical long-range interaction potentials 
on a 3D lattice, or for generally distributed potentials in a volume is one of the challenges in 
the numerical treatment of many-body systems in molecular dynamics,  quantum chemical computations,
simulations of proteins and large solvated biological systems 
\cite{CRYSCOR:12,DesHolm:98,SaDoWeKu:14} and in stochastic computations
\cite{DolKhLitMat:14}.
Mathematical aspects of the problems arising in modeling of periodic and quasi-periodic systems
have been considered in \cite{CanLe:2013,CanEhrMad:2012,LSJ:2010,OrtnEhrl:13,LuOrtnerVK:2013}.  
Beginning with the widely spread Ewald summation techniques \cite{Ewald:27},
the development of lattice-sum methods has led to well established 
algorithms for numerical evaluation of long-range interaction potentials of 
large multiparticle systems, see for example 
\cite{DYP:93,KuScuser:04,PolGlo:96,TouBo_Ewald:96,Hune_Ewald:99} and references therein. 
These methods usually combine the original Ewald summation approach 
with the fast Fourier transform (FFT) or fast multipole method \cite{RochGreen:87}. 
The fast multipole method is well suited for summation of non-uniformly distributed  
potentials, making benefits from direct approximation of closely positioned source functions
and clustered summation of far fields.
The numerical complexity of the Ewald-type computational schemes scales at least linearly 
in the total number of potentials, $O(L^3)$, distributed on the $L\times L \times L$ lattice.

In \cite{VeBoKh:Ewald:14} the new generation of grid-based lattice summation techniques
for long-range interaction potentials on rectangular lattices is introduced,
which is based  on the idea of assembling the directional vectors in the 
low-rank canonical tensor format. 
This tensor approach provides the efficient summation of a large number of
potentials on a 3D lattice with complexity scaling $O(L)$ instead of $O(L^3)$. 

This paper presents a significant generalization of the previous approach \cite{VeBoKh:Ewald:14}
to the case of 3D lattices with defects, such as 
vacancies, impurities and non-rectangular geometries of lattice points, 
as well as in the case of hexagonal symmetries.
Here both the Tucker and canonical tensor formats are employed.
The single potential function in 3D, sampled on a large $N\times N \times N$ representation grid 
in a bounding box, is approximated  with a guaranteed precision
by a low-rank  Tucker/canonical reference tensor. This tensor provides the values
of the discretized potential at any point of this fine auxiliary 3D grid,
but needs only $O(N)$ storage. 
Then each 3D singular kernel function involved in the summation is represented 
on the same grid by a shift of the reference tensor along lattice vector.
Directional vectors of the Tucker/canonical tensor defining a full lattice sum  
are assembled by the 1D summation of the corresponding skeleton vectors
for the shifted tensor. 
In the case of 3D cubic $L\times L \times L$ lattice
the separation ranks of the resultant sum are proven to be the same as for the reference tensor. 
The required storage scales linearly in the 1D grid-size, 
$O(N)$, while the numerical cost is estimated by $O(N L)$.
The lattice nodes are not required to exactly
coincide with the grid points of the global $N\times N\times N$ representation grid 
since the accuracy of the resulting tensor sum
is well controlled due to easy availability of large grid size $N$.

The low-rank tensor approximation to the spherically symmetric reference potential 
is based on the separable representation of the analytic kernel function
by using its integral Laplace transform. 
In particular, the algorithm in \cite{BeHaKh:08} based on the ${sinc}$-quadrature approximation 
to the Laplace transform of the Newton kernel function $\frac{1}{r}$
(see \cite{Braess:95,HaKhtens:04I,GaHaKh3:05}) is applied. 
Literature surveys on the most commonly used in computational practice
tensor formats like canonical, Tucker and matrix product states (or tensor train) representations, 
as well as on basics of multilinear algebra and the recent tensor numerical methods 
for solving PDEs, can be found in \cite{Kolda,Scholl:11,Osel_TT:11,GraKresTo:13,KhorSurv:14,HaSchneid_Surv:15}
(see also \cite{VeKh_Diss:10} and \cite{Dolg_PhD:14}).

In the case of defected lattices, the overall potential is obtained as  an algebraic sum  
of several tensors, each of which represents the contribution of certain cluster of individual
defects that leads to increase in the tensor rank of the resulting potential sum.
For rank reduction in the canonical format the canonical-to-Tucker decomposition 
is applied based on the reduced higher-order SVD (HOSVD) approximation introduced in \cite{KhKh3:08}.
Here we generalize the reduced HOSVD (RHOSVD) approximation to the cases of 
Tucker input tensors\footnote{See \cite{DMV-SIAM2:00} concerning the notion of the initial HOSVD scheme.}. 
We formulate stability conditions and prove the error bounds for the RHOSVD approximation 
to a sum of canonical/Tucker tensors.
In particular, the RHOSVD scheme was successfully applied to the \it{direct} summation 
of electrostatic potentials of nuclei in a molecule \cite{KhorVBAndrae:11} 
for calculation of the one-electron integrals in the framework of
3D grid-based Hartree-Fock solver by tensor-structured methods \cite{VKH_solver:13}.
In general, the direct summation of canonical/Tucker tensors accomplished by
the RHOSVD-type rank reduction proves to be efficient in the case of rather arbitrary
positions of a moderate number of potentials (like nuclei in a single molecule).

Thus, the  canonical/Tucker tensor representation of the lattice sum of 
interaction potentials in the presence of defects can be computed with high accuracy, 
and in a completely algebraic way. 
The tensor approach is advantageous in further functional calculus
with the lattice potential sums represented on a 3D grid, like integration or differentiation,
using tensor arithmetics of 1D complexity \cite{KhKh3:08,VeKh_Diss:10}. 
Notice that the summation cost in the Tucker/canonical formats,
$O(L \, N)$, can be reduced to the logarithmic scale in the grid size, $O(L \log N )$, 
by using the low-rank quantized tensor approximation (QTT), see \cite{KhQuant:09},
of long canonical/Tucker vectors as it was suggested and analyzed in \cite{VeBoKh:Ewald:14}.

The presented approach yields enormous reduction in storage and computing time.
Our numerical tests show that summation of two millions of potentials on a 
3D lattice on a grid of size $10^{15}$ takes about 15 seconds in Matlab implementation.
Generally, this concept originates from numerical studies in \cite{KhKh:06,VeKh_Diss:10}
which displayed  that the Tucker tensor rank of the 3D lattice sum of discretized 
Slater functions is close to the rank of a single Slater potential.
The approach applies to a general class of kernel functions  including those for
the Newton, Slater, Yukawa, Lennard-Jones, and dipole-dipole interactions.
It is can be efficient for calculation of electronic properties of large finite 
crystalline systems like quantum dots, which are 
intermediate between bulk (periodic) systems and discrete molecules.



The rest of the paper is structured as following.
\S2 discusses the 3D grid-based rank-structured canonical/Tucker tensor representations 
to a single kernel based on the approximation properties of tensor decompositions
to a class of spherically symmetric analytic functions. Section \S3 describes the 
direct tensor calculation of a sum of the shifted single potentials and  
focuses on the construction and analysis of the algorithms of
assembled Tucker tensor summation of the non-local potentials on a rectangular 3D lattice.
\S4 describes the Tucker/canonical summation method for lattices with defects
and different geometries. In this case,
the rank optimization is discussed, and the error bound for the generalized RHOSVD approximation
in the Tucker format is proved.
In particular, \S4.3 outlines the extension of the tensor-based lattice
summation techniques  to the class of non-rectangular lattices or rather general shape of 
the set of active lattice points (say, multilevel step-type boundaries).
Conclusions summarize the main features of the approach and outlines
the further perspectives.

\section{Tensor decomposition for analytic potentials} 
\label{sec:NewtTens}

Methods of separable approximation to the 3D Newton kernel (electrostatic potential) 
using the Gaussian sums have been addressed in the chemical and mathematical literature 
since \cite{Boys:56} and \cite{Braess:BookApTh,Braess:95,HaKhtens:04I,GaHaKh3:05,BeHaKh:08}, 
respectively. For the readers convenience, in this section, we recall the main 
ingredients of the tensor approximation scheme for classical potentials.

\subsection{Grid-based canonical/Tucker representation of a single kernel} 
\label{ssec:CoulombUnit}

We discuss the grid-based method for the low-rank canonical and Tucker tensor 
representations of a spherically symmetric kernel function $p(\|x\|)$, 
$x\in \mathbb{R}^d$ for $d=1,2,3$
(for example, for the 3D Newton we have $p(\|x\|)=\frac{1}{\|x\|}$, $x\in \mathbb{R}^3$)
by its projection onto the set
of piecewise constant basis functions, see \cite{BeHaKh:08} for more details.


In the computational domain  $\Omega=[-b/2,b/2]^3$, 
let us introduce the uniform $n \times n \times n$ rectangular Cartesian grid $\Omega_{n}$
with the mesh size $h=b/n$.
Let $\{ \psi_\textbf{i}\}$ be a set of tensor-product piecewise constant basis functions,
$  \psi_\textbf{i}(\textbf{x})=\prod_{\ell=1}^d \psi_{i_\ell}^{(\ell)}(x_\ell)$,
for the $3$-tuple index $\tb{i}=(i_1,i_2,i_3)$, $i_\ell \in \pl{1,...,n}$, $\ell=1,\, 2,\, 3 $.
The kernel $p(\|x\|)$ can be discretized by its projection onto the basis set $\{ \psi_\textbf{i}\}$
in the form of a third order tensor of size $n\times n \times n$, defined pointwise as
\begin{eqnarray}
\mathbf{P}:=\pc{p_\tb{i}} \in \mathbb{R}^{n\times n \times n},  \quad
 p_\tb{i} = 
\int_{\mathbb{R}^3} {\psi_{\tb{i}}({x})}p({\|{x}\|}) \,\, \mathrm{d}{x}.
  \label{galten}
\end{eqnarray}

The low-rank canonical decomposition of the $3$rd order tensor $\mathbf{P}$ is based 
on using exponentially convergent 
$\operatorname*{sinc}$-quadratures for approximation of the Laplace-Gauss transform 
to the analytic function $p(z)$ specified by certain weight $a(t) >0$,
\begin{align} \label{eqn:laplace} 
p(z)=\int_{\mathbb{R}_+} a(t) e^{- t^2 z^2} \,\mathrm{d}t \approx
\sum_{k=-M}^{M} a_k e^{- t_k^2 z^2} \quad \mbox{for} \quad |z| > 0,
\end{align} 
where the quadrature points and weights are given by 
\begin{equation} \label{eqn:hM}
t_k=k \mathfrak{h}_M , \quad a_k=a(t_k) \mathfrak{h}_M, \quad 
\mathfrak{h}_M=C_0 \log(M)/M , \quad C_0>0.
\end{equation}
Under the assumption $0< a \leq \|z \|  < \infty$
this quadrature can be proven to provide the exponential convergence rate in $M$
for a class of analytic functions $p(z)$, see \cite{Stenger,HaKhtens:04I,Khor_CVS:07}. 
For example, in the particular case $p(z)=1/z$,
which can be adapted to the Newton kernel by substitution $z=\sqrt{x_1^2 + x_2^2  + x_3^2}$,
we apply the Laplace-Gauss transform
\[
 \frac{1}{z}= \frac{2}{\sqrt{\pi}}\int_{\mathbb{R}_+} e^{- t^2 z^2 } dt.
\]
We proceed with further discussion of this issue in \S\ref{ssec:GeneraKern}. 

Now for any fixed $x=(x_1,x_2,x_3)\in \mathbb{R}^3$, 
such that $\|{x}\| > 0$, 
we apply the $\operatorname*{sinc}$-quadrature approximation to obtain the separable 
expansion
\begin{equation} \label{eqn:sinc_Newt}
 p({\|{x}\|}) =   \int_{\mathbb{R}_+} a(t)
e^{- t^2\|{x}\|^2} \,\mathrm{d}t  \approx 
\sum_{k=-M}^{M} a_k e^{- t_k^2\|{x}\|^2}= 
\sum_{k=-M}^{M} a_k  \prod_{\ell=1}^3 e^{-t_k^2 x_\ell^2}.
\end{equation}
Under the assumption $0< a \leq \|{x}\| \leq A < \infty$
this approximation provides the exponential convergence rate in $M$,
\begin{equation} \label{sinc_conv}
\left|p({\|{x}\|}) - \sum_{k=-M}^{M} a_k e^{- t_k^2\|{x}\|^2} \right|  
\le \frac{C}{a}\, \displaystyle{e}^{-\beta \sqrt{M}},  
\quad \text{with some} \ C,\beta >0.
\end{equation}
Combining \eqref{galten} and \eqref{eqn:sinc_Newt}, and taking into account the 
separability of the Gaussian basis functions, we arrive at the low-rank 
approximation to each entry of the tensor $\mathbf{P}$,
\begin{equation*} \label{eqn:C_nD_0}
 p_\tb{i} \approx \sum_{k=-M}^{M} a_k   \int_{\mathbb{R}^3}
 \psi_\tb{i}(\textbf{x}) e^{- t_k^2\|{x}\|^2} \mathrm{d} {x}
=  \sum_{k=-M}^{M} a_k  \prod_{\ell=1}^{3}  \int_{\mathbb{R}}
\psi^{(\ell)}_{i_\ell}(x_\ell) e^{- t_k^2 x^2_\ell } \mathrm{d}x_\ell.
\end{equation*}
Define the vector (recall that $a_k >0$) 
$\textbf{p}^{(\ell)}_k = a_k^{1/3} {\bf b}^{(\ell)}(t_k) \in \mathbb{R}^{n_\ell}$, where
\begin{equation*} \label{eqn:galten_int}
{\bf b}^{(\ell)}(t_k)
= \left[b^{(\ell)}_{i_\ell}(t_k)\right]_{i_\ell=1}^{n_\ell} \in \mathbb{R}^{n_\ell}
\quad \text{with } \quad b^{(\ell)}_{i_\ell}(t_k)= 
\int_{\mathbb{R}} \psi^{(\ell)}_{i_\ell}(x_\ell) e^{- t_k^2 x^2_\ell } \mathrm{d}x_\ell,
\end{equation*}
then the $3$rd order tensor $\mathbf{P}$ can be approximated by 
the $R$-term canonical representation
\begin{equation} \label{eqn:sinc_general}
    \mathbf{P} \approx  \mathbf{P}_R =
\sum_{k=-M}^{M} a_k \bigotimes_{\ell=1}^{3}  {\bf b}^{(\ell)}(t_k)
= \sum\limits_{q=1}^{R} {\bf p}^{(1)}_q \otimes {\bf p}^{(2)}_q \otimes {\bf p}^{(3)}_q
\in \mathbb{R}^{n\times n \times n},
\end{equation}
where $R=2M+1$. For the given threshold $\varepsilon >0 $,  $M$ is chosen as the minimal number
such that in the max-norm
\begin{equation*} \label{eqn:error_control}
\| \mathbf{P} - \mathbf{P}_R \|  \le \varepsilon \| \mathbf{P}\|.
\end{equation*}
The canonical vectors are renumbered by $k \to q=k+M+1$, 
${\bf p}^{(\ell)}_q ={\bf p}^{(\ell)}_{k} \in \mathbb{R}^n$, $\ell=1, 2, 3$.
The canonical tensor ${\bf P}_{R}$ in (\ref{eqn:sinc_general})
approximates the discretized 3D symmetric kernel function 
$p({\|x\|})$ ($x\in \Omega$), 
centered at the origin, such that ${\bf p}^{(1)}_q={\bf p}^{(2)}_q={\bf p}^{(3)}_q$ ($q=1,...,R$).

In the following, we also consider a Tucker approximation of the $3$rd order tensor ${\bf P}$.
Given rank parameters ${\bf r}=(r_1,r_2,r_3)$, the set of rank-${\bf r}$ Tucker tensors
(the Tucker format) is defined by the following parametrization, 
${\bf T}=[ t_{{i_1}{i_2}{i_3}} ]\in \mathbb{R}^{n \times n \times n}$ (${i_\ell} \in \{1,...,n\}$),
\begin{equation} \label{eqn:Tucker_singePot}
 {\bf T}: = \sum\limits_{{\bf k}= {\bf 1}}^{\bf r} b_{\bf k} 
{\bf t}^{(1)}_{k_1}  \otimes {\bf t}^{(2)}_{k_2}\otimes {\bf t}^{(3)}_{k_3}
\equiv {\bf B}\times_1 { T}^{(1)}\times_2 {T}^{(2)}\times_3 {T}^{(3)}, 
\end{equation}
where the orthogonal side-matrices 
${T}^{(\ell)}=[{\bf t}^{(\ell)}_1 ... {\bf t}^{(\ell)}_{r_\ell}] 
\in \mathbb{R}^{n \times r_\ell}$, 
$\ell=1,2,3$, define the set of Tucker vectors. 
Here $\times_\ell$ means the contracted product a tensor with a vector, 
and ${\bf B}\in \mathbb{R}^{r_1 \times r_2 \times r_3} $ is the core coefficients tensor.
Choose the truncation error $\varepsilon >0$ for the canonical approximation ${\bf P}_{R}$
obtained by the quadrature method, then compute the best orthogonal Tucker 
approximation of ${\bf P}$ with tolerance $O(\varepsilon)$ by applying 
the canonical-to-Tucker algorithm \cite{KhKh3:08} to the canonical 
tensor ${\bf P}_{R}\mapsto {\bf T}_{\bf r} $. 
The latter algorithm is based on the rank optimization via ALS iteration.  
The rank parameters ${\bf r}$ of the resultant Tucker approximand ${\bf T}_{\bf r}$
is minimized subject to the $\varepsilon$-error control,
\begin{equation*} \label{eqn:Tuckerror_contr}
\| \mathbf{P}_{R} - \mathbf{T}_{\bf r} \|  \le \varepsilon \| \mathbf{P}_{R}\|.
\end{equation*}
\begin{remark}\label{rem:TuckerRank}
Since the maximal Tucker rank does not exceed the canonical one we apply the approximation results
for canonical tensor to derive  the exponential convergence in Tucker rank 
for the wide class of functions $p$. This implies the relation $\max\{ r_\ell\} =O(|\log \varepsilon|^2)$
which can be observed in all numerical test implemented so far.
\end{remark}

\begin{figure}[htbp]
\centering
\includegraphics[width=7.0cm]{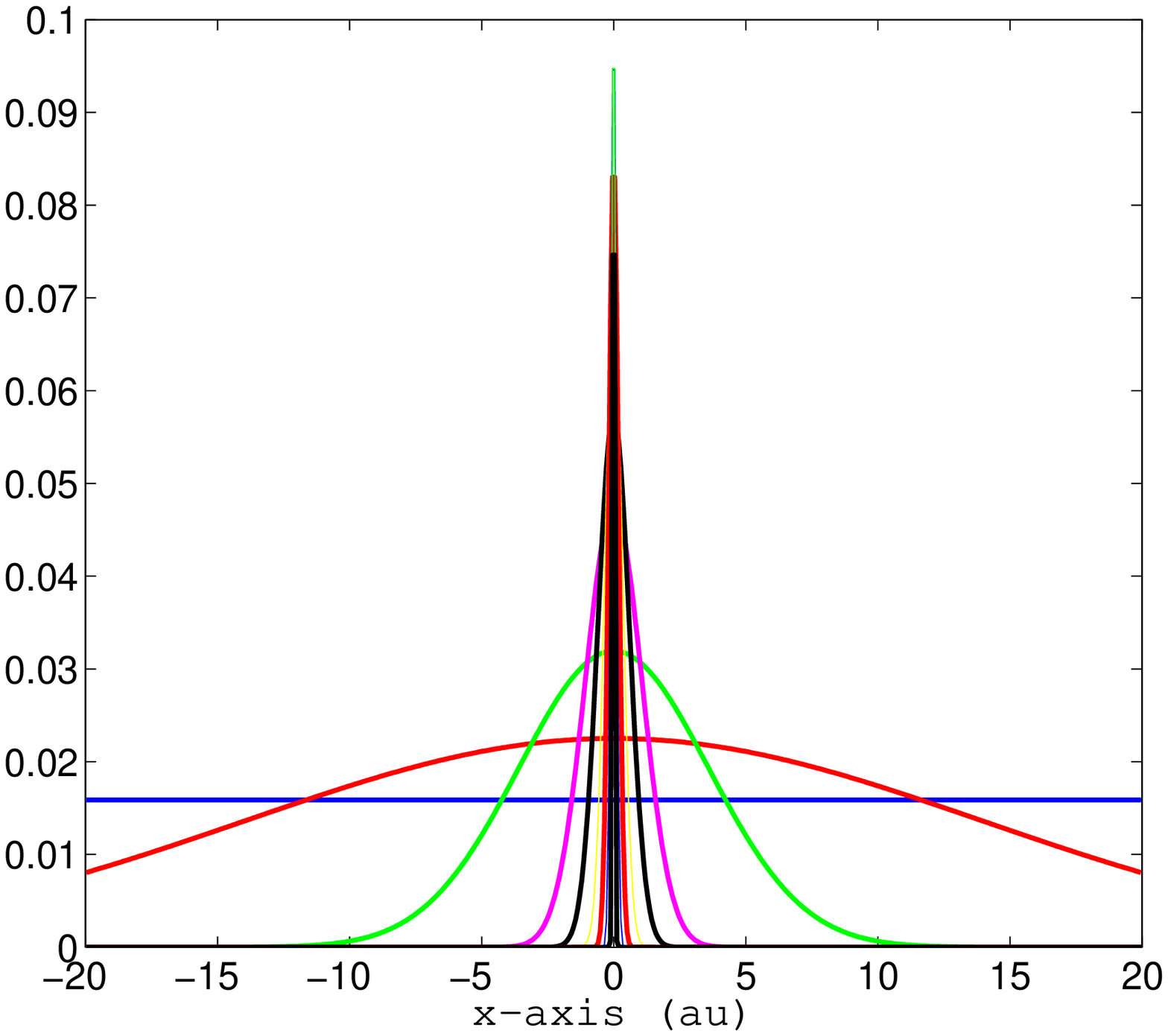}
\includegraphics[width=7.0cm]{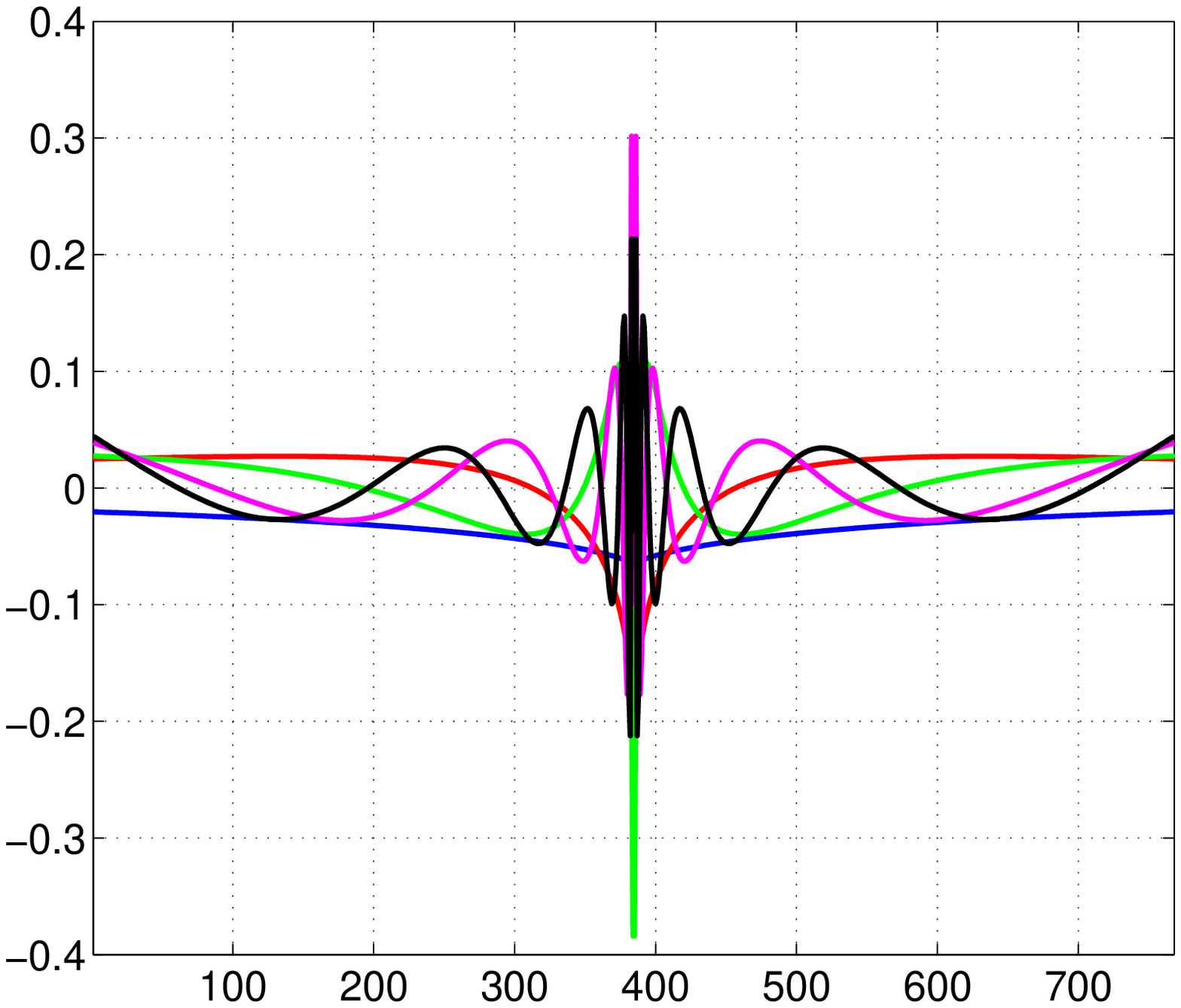}
\caption{Vectors  of the canonical $\{{\bf p}^{(1)}_q\}_{q=1}^R$ (left) 
and Tucker $\{{\bf t}^{(1)}_k\}_{k=1}^{r_1}$ (right) tensor representations 
for the single Newton kernel displayed along  $x$-axis.}
\label{fig:Newton}  
\end{figure}

Figure \ref{fig:Newton} displays several vectors of the canonical and Tucker tensor 
representations 
for a single Newton kernel along $x$-axis from a set $\{P^{(1)}_q\}_{q=1}^R$.
Symmetry of the tensor ${\bf P}_{R}$ implies that the canonical vectors ${\bf p}^{(2)}_q$ and ${\bf p}^{(3)}_q$ 
corresponding to $y$ and $z$-axes, respectively, are of the same shape as ${\bf p}^{(1)}_q$.
It is clearly seen that there are canonical/Tucker vectors representing the long-, 
intermediate- and short-range contributions to the total electrostatic potential. This interesting feature 
will be also recognized for the low-rank lattice sum of potentials 
(see \S \ref{ssec:Can2Tuck_lattice}).

Table \ref{Table_Times} presents CPU times (sec) for generating a canonical rank-$R$ tensor 
approximation of the single Newton kernel over $n\times n\times n $ 3D Cartesian grid, 
corresponding to Matlab implementation on a terminal of the 8 AMD Opteron Dual-Core processor.
The corresponding mesh sizes are given in {\it Angstroms}. 
We observe a logarithmic scaling of the canonical rank $R$ in the grid size $n$, while the maximal 
Tucker rank has the tendency to decrease for larger $n$.
The compression rate for the grid $73768^3$, that is the ratio $n^3/(n R)$ for the 
canonical format and $n^3/(3r^3 n )$ for the Tucker format are of the order of
$10^8$ and $10^7$, respectively. 

\begin{table}[htb]
\begin{center}%
\begin{tabular}
[c]{|r|r|r|r|r|r|}%
\hline
grid size $n^3 $   & $4608^3$  & $9216^3$ & $18432^3$ & $36864^3$   & $73768^3$\\
 \hline
mesh size $h\, (\AA{})$    & $0.0019$  & $0.001$ & $4.9\cdot 10^{-4}$ & $2.8\cdot 10^{-4}$
& $1.2\cdot 10^{-4}$\\
 \hline  \hline
 Time (Canon.)   &  $2.$    &   $2.7$  &   $8.1 $ &  $38$    & $164$ \\
 \hline
Canonical rank $R$ &  $34$ &   $37$  &   $39 $   & $41$    & $43$ \\
 \hline \hline
Time (C2T)   &  $17$    &   $38$  &   $85 $ &  $200$    & $435$ \\
 \hline
Tucker rank   &  $12$   &   $11$  &   $10$   & $8$ & $6$ \\
 \hline
 \end{tabular}
\caption{CPU times (Matlab) to compute with tolerance $\varepsilon = 10^{-6}$
canonical and Tucker vectors of ${\bf P}_{R}$ for the single Newton kernel in a box. }
\label{Table_Times}
\end{center}
\end{table}

Notice that the low-rank canonical/Tucker approximation of the tensor ${\bf P}$ is 
the problem independent task, hence the respective canonical/Tucker vectors can be precomputed at once
on  large enough 3D $n\times n\times n $ grid, and then stored for the multiple use. 
The storage size is bounded by $R n$ or $3 r n + r^3$.

\subsection{Low-rank representation for the general class of kernels} 
\label{ssec:GeneraKern}

Along with Coulombic systems corresponding to $p(\|x\|)= \frac{1}{\|x\|}$,
the tensor approximation described above can be also applied to a wide class of commonly 
used long-range kernels $p(\|x\|)$ in $\mathbb{R}^3$, for example, 
to the Slater, Yukawa, 
Lennard-Jones or Van der Waals and dipole-dipole interactions potentials 
defined as follows, 
\[
 \mbox{Slater function:} \quad p(\|x\|)=\exp(- \lambda {\|x\|}),\quad \lambda >0,
\]
\[
 \mbox{Yukawa kernel:} \quad p(\|x\|)= \frac{\exp(- \lambda {\|x\|})}{\|x\|},\quad \lambda >0,
\]
\[
\mbox{Lennard-Jones potential:} \quad p(\|x\|)= 4 \epsilon 
\left[\left(\frac{\sigma }{\|x\|}\right)^{12} - \left(\frac{\sigma }{\|x\|}\right)^{6}  \right],
\]
The simplified version of the Lennard-Jones potential is the so-called Buckingham function
\[
\mbox{Buckingham potential:} \quad p(\|x\|)= 4 \epsilon 
\left[e^{\|x\|/r_{0}} - \left(\frac{\sigma }{\|x\|}\right)^{6}  \right].
\]
The electrostatic potential energy for the dipole-dipole interaction due to Van der Waals forces 
is defined by
\[
 \mbox{Dipole-dipole interaction energy:} \quad p(\|x\|)= \frac{C_0}{\|x\|^3}. 
\]
The quasi-optimal low-rank decompositions based on the $sinc$-quadrature approximation 
to the Laplace transforms of the above mentioned functions can be rigorously proven for
a wide class of generating kernels. 
In particular, the  following Laplace (or Laplace-Gauss) integral transforms \cite{Zeidler:03}
with a parameter $\rho >0$ can be applied for the $sinc$-quadrature approximation of the 
above mentioned functions,
\begin{eqnarray} 
\label{eqn:LalpSlater}
 e^{-2 \sqrt{\kappa\rho}}&=& \frac{\sqrt{\kappa}}{\sqrt{\pi}}\int_{\mathbb{R}_+}
 t^{-3/2}e^{-\kappa/t} \, e^{ - \rho t } dt,\\
\label{eqn:LalpYukawa}
 \frac{e^{-\kappa \sqrt{\rho}}}{\sqrt{\rho}}&=& \frac{2}{\sqrt{\pi}}\int_{\mathbb{R}_+}
e^{- \kappa^2/t^2}\, e^{-\rho t^2 } dt,\\
\label{eqn:LalpNewt}
 \frac{1}{\sqrt{\rho}}&=& \frac{2}{\sqrt{\pi}}\int_{\mathbb{R}_+} e^{-\rho t^2 } dt,\\
\label{eqn:LalpPoly}
 \frac{1}{{\rho}^n}&=& \frac{1}{(n-1)!}\int_{\mathbb{R}_+}t^{n-1} e^{-\rho t } dt,\quad n=1,2,...
\end{eqnarray}
combined with the subsequent substitution of a parameter $\rho$ by the appropriate 
function $\rho(x)=\rho(x_1,x_2,x_3)$, usually by using an
additive representation $\rho=c_1 x_1^p + c_2 x_2^q  + c_3 x_3^z$.
In the cases (\ref{eqn:LalpPoly}) ($n=1$) and (\ref{eqn:LalpNewt}) 
the convergence rate for the $sinc$-quadrature approximations of type (\ref{eqn:hM}) 
has been considered in \cite{Braess:BookApTh,Braess:95} and 
later analyzed in more detail in \cite{GaHaKh3:05,HaKhtens:04I}.
The case of the Yukawa and Slater kernel has been investigated in 
\cite{Khor:06,Khor_CVS:07}. The exponential error bound for the 
general transform (\ref{eqn:LalpPoly}) can be derived by minor modifications of 
the above mentioned results.

\begin{remark}\label{rem:Lapl_latticeSum}
The idea behind the low-rank tensor representation for a sum of spherically symmetric potentials 
on a 3D lattice can be already recognized on the continuous level by 
introducing the Laplace transform of the generating kernel.
For example, in representation (\ref{eqn:LalpYukawa}) with the particular choice $\kappa=0$,
that is given by (\ref{eqn:LalpNewt}), we can set up $\rho=x_1^2 + x_2^2  + x_3^2$, i.e. 
$p(\|x\|)=1/\|x\|$, ($1 \leq x_\ell < \infty $), and apply 
the $sinc$-quadrature approximation as in (\ref{eqn:laplace})-(\ref{eqn:hM}),
\begin{align} \label{eqn:laplaceNewt} 
p(z)=\frac{2}{\sqrt{\pi}}\int_{\mathbb{R}_+}  e^{- t^2 z^2} \,\mathrm{d}t \approx
\sum_{k=-M}^{M} a_k e^{- t_k^2 z^2} \quad \mbox{for} \quad |z| > 0.
\end{align} 
Now the simple sum on a rectangular lattice of width $b>0$,
\[
\Sigma_{L}(x)= \sum\limits_{i_1,i_2,i_3=1}^L \frac{1}{\sqrt{(x_1+i_1 b)^2 + (x_2+i_2 b)^2  + (x_3+i_3 b)^2}},
\]
can be represented by the agglomerated integral transform
 \begin{equation} \label{eqn:WindowCan_sumExp}
\begin{split}
  \Sigma_{L}(x)& = \frac{2}{\sqrt{\pi}}\int_{\mathbb{R}_+}[ \sum\limits_{i_1,i_2,i_3=1}^L  
e^{-[(x_1+i_1 b)^2 + (x_2+i_2 b)^2  + (x_3+i_3 b)^2] t^2 }] dt \\
             & =\frac{2}{\sqrt{\pi}}\int_{\mathbb{R}_+}
\sum\limits_{k_1=1}^L e^{-(x_1+k_1 b)^2 t}\sum\limits_{k_2=1}^L e^{-(x_2+k_2 b)^2 t}
\sum\limits_{k_3=1}^L e^{-(x_3+k_3 b)^2 t} dt, 
\end{split}
\end{equation}
where the integrand is separable. Representation (\ref{eqn:WindowCan_sumExp}) indicates that applying 
the same quadrature approximation 
to the lattice sum integral (\ref{eqn:WindowCan_sumExp}) as that for the single 
kernel (\ref{eqn:laplaceNewt}) will lead to the decomposition of the total sum
of potentials with the same canonical rank as for the single one.
\end{remark}

In the following, we construct the low-rank canonical and Tucker 
decompositions of the lattice sum of interaction potentials 
discretized on the fine representation 3D-grid and applied to the 
general class of kernel functions and more general configuration of a lattice.

%

\section{Tucker decomposition for lattice sum of potentials} \label{sec:TensSum}

\subsection{Direct tensor sum for a moderate number of arbitrarily distributed potentials}
\label{ssec:sum_unit}

In this paragraph, we recall the direct tensor summation of the electrostatic potentials
for a moderate number of arbitrarily distributed sources 
as introduced in \cite{KhorVBAndrae:11,VKH_solver:13}. 

The basic example in electronic structure calculations 
is concerned with the nuclear potential operator 
describing the Coulombic interaction of electrons with the nuclei in a molecular 
system in a box corresponding to the choice $p(\|{x} \|)=\frac{1}{\|{x} \|}$.
We consider a function $v_c(x)$ describing the interaction potential of 
several nuclei
in a computational box 
$\Omega = [-b/2,b/2]^3 \subset \mathbb{R}^3$,  
\begin{equation}\label{eqn:V_c_GenerP}
 v_c(x)=  \sum_{\nu=1}^{M_0} Z_\nu p(\|{x}-a_\nu \|),\quad
Z_\nu >0, \;\; x,a_\nu \in \Omega, 
\end{equation}
where $M_0$ is the (moderate) number of nuclei in $\Omega$, and $a_\nu$,  $Z_\nu >0$, 
represent their coordinates and ``charges``, respectively.
We are interested in the low-lank representation of 
the projected tensor ${\bf V}_c$ along the line of \S\ref{ssec:CoulombUnit},
\[
\mathbf{V}_c:=\left[\int_{\mathbb{R}^3} {\psi_{\tb{i}}({x})}v_c({x}) \,\, \mathrm{d}{x} \right]
\in \mathbb{R}^{n\times n \times n}.
\]
Similar to \cite{VKH_solver:13,VeBoKh:Ewald:14}, we first approximate the non-shifted kernel 
$p({\|x\|})$ on the auxiliary 
extended box $\widetilde{\Omega}=[-b,b]^3 $ in the canonical format by its projection onto the basis set
$\{ \psi_\textbf{i}\}$ of piecewise constant functions as described in \S\ref{ssec:CoulombUnit}, and
defined on a $2n\times 2n \times 2n$ uniform tensor grid $\widetilde{\Omega}_{2n}$ with
the mesh size $h$, with embedding ${\Omega}_{n} \subset \widetilde{\Omega}_{2n}$.
This defines  the ''reference`` rank-$R$ canonical tensor as above 
\begin{equation} \label{master_pot}
\widetilde{\bf P}_R= 
\sum\limits_{q=1}^{R} \widetilde{\bf p}^{(1)}_q \otimes \widetilde{\bf p}^{(2)}_q \otimes \widetilde{\bf p}^{(3)}_q
\in \mathbb{R}^{2n\times 2n \times 2n}.
\end{equation}

For ease of exposition, we assume  that each nuclei coordinate $a_\nu$ is located
exactly\footnote{Our numerical scheme is designed for nuclei positioned arbitrarily  
in the computational box where approximation error of order $O(h)$ is controlled by choosing 
large enough grid size $n$.
Indeed, $1D$ computational cost enables us usage of fine grids of size $n^3\approx 10^{15}$ 
in Matlab implementation, yielding mesh size $h \approx 10^{-4} \div 10^{-5}$ $\AA{}$, i.e. 
$h$ is of the order of the atomic radii.
This grid-based tensor calculation scheme for the nuclear potential operator was tested numerically
in molecular calculations \cite{KhorVBAndrae:11}, where it was compared with the results of 
analytical evaluation of the same operator from benchmark quantum chemical packages.} 
at certain grid-point 
$a_\nu=(i_\nu h-b/2,j_\nu h-b/2,k_\nu h-b/2)$, with some $1 \leq i_\nu,j_\nu,k_\nu \leq n $.
Now we are in a position to introduce the rank-$1$ shift-and-windowing operator  
$$
{\cal W}_{\nu}={\cal W}_{\nu}^{(1)}\otimes {\cal W}_{\nu}^{(2)}\otimes {\cal W}_{\nu}^{(3)}:
\mathbb{R}^{2n\times 2n \times 2n} \to \mathbb{R}^{n\times n \times n}, \quad 
{\mbox for}\quad \nu=1,...,M_0,
$$ 
via 
\begin{equation} \label{eqn:sub_tens}
 {\cal W}_{\nu} \widetilde{\bf P}_R :=
\widetilde{\bf P}_R(i_\nu +n/2:i_\nu +3/2n;j_\nu +n/2 :j_\nu +3/2 n;k_\nu +n/2:k_\nu +3/2n)
\in \mathbb{R}^{n\times n \times n}. 
\end{equation}
With this notation, the projected tensor ${\bf V}_c$ approximating 
the total electrostatic potentials $v_c(x)$ in $\Omega$ is represented by a 
direct sum of low-rank canonical tensors
\begin{equation} \label{eqn:WindowCan_sum}
\begin{split}
{\bf V}_c \mapsto {\bf P}_{c} & = \sum_{\nu=1}^{M_0} Z_\nu {\cal W}_{\nu} \widetilde{\bf P}_R \\
             & =\sum_{\nu=1}^{M_0} Z_\nu  
\sum\limits_{q=1}^{R} {\cal W}_{\nu}^{(1)} \widetilde{\bf p}^{(1)}_q \otimes 
{\cal W}_{\nu}^{(2)} \widetilde{\bf p}^{(2)}_q 
\otimes {\cal W}_{\nu}^{(3)} \widetilde{\bf p}^{(3)}_q\in \mathbb{R}^{n\times n \times n},
\end{split}
\end{equation}
where every rank-$R$ canonical tensor 
${\cal W}_{\nu} \widetilde{\bf P}_R \in \mathbb{R}^{n\times n \times n}$
is thought as a sub-tensor of the reference tensor
$\widetilde{\bf P}_R \in \mathbb{R}^{2n\times 2n \times 2n}$ obtained 
by its shifting and restriction (windowing) onto the $n \times n \times n$ grid 
in the computational box $\Omega_{n} \subset \widetilde{\Omega}_{2n}$. Here a shift
from the origin is specified according to the coordinates of the corresponding nuclei, $a_\nu$,
counted in the $h$-units. 

For example, the electrostatic potential centered at the origin, i.e. with $a_\nu=0$, corresponds
to the restriction of $\widetilde{\bf P}_R\in \mathbb{R}^{2n\times 2n \times 2n}$ 
onto the initial computational box $\Omega_{n}$, i.e. onto the index set (assume that $n$ is even)
$$
{\cal I}_0 = \{(n/2+i,n/2+j,n/2+k): \; i,j,k\in \{1,...,n\}\}.
$$

The projected tensor ${\bf V}_c$ approximating  the function in (\ref{eqn:V_c_GenerP}) is represented 
as a canonical tensor ${\bf P}_{c}$ with the rough  
bound on its rank $R_{c}=rank({\bf P}_{c}) \leq M_0 R $, where $R= rank(\widetilde{\bf P}_{R})$.
However, our numerical tests for moderate size molecules indicate that  
the tensor ranks of the $(M_0 R)$-term canonical sum representing 
${\bf P}_{c}$ can be considerably reduced, such that $R_c \approx R$.
This rank optimization can be implemented, for example, by the multigrid version of the 
canonical rank reduction algorithm, canonical-Tucker-canonical, based on RHOSVD approximation \cite{KhKh3:08}.
The resultant canonical tensor will be denoted by ${\bf P}_{R_c}$.

Along the same line, the direct sum in the Tucker format can be represented by using 
shift-and-windowing projection of the ''reference'' rank-${\bf r}$ Tucker tensor 
\begin{equation} \label{eqn:TuckRef_sum}
\widetilde{\bf T}_{{\bf r}}:= 
\sum\limits_{{\bf k}= {\bf 1}}^{\bf r} b_{\bf k}\widetilde{\bf t}^{(1)}_{k_1} \otimes
\widetilde{\bf t}^{(2)}_{k_2} \otimes \widetilde{\bf t}^{(3)}_{k_3}
\in \mathbb{R}^{2n\times 2n \times 2n},
\end{equation}
approximating the Newton kernel in the Tucker format,
\begin{equation} \label{eqn:WindowTuck_sum}
\begin{split}
{\bf V}_c \mapsto {\bf T}_{c} & = \sum_{\nu=1}^{M_0} Z_\nu {\cal W}_{\nu} \widetilde{\bf T}_{{\bf r}} \\
             & =\sum_{\nu=1}^{M_0} Z_\nu \sum\limits_{{\bf k}= {\bf 1}}^{\bf r} b_{\bf k}
{\cal W}_{\nu}^{(1)} \widetilde{\bf t}^{(1)}_{k_1} \otimes {\cal W}_{\nu}^{(2)} \widetilde{\bf t}^{(2)}_{k_2} 
\otimes {\cal W}_{\nu}^{(3)} \widetilde{\bf t}^{(3)}_{k_3} \in \mathbb{R}^{n\times n \times n},
\end{split}
\end{equation}
As in the case of canonical decomposition,
the rank reduction procedure based on ALS-type iteration applies to the sum of 
Tucker tensors, ${\bf T}_{c}$, resulting in the optimized Tucker tensor 
${\bf T}_{{\bf r}_c}$ with the reduced rank parameter ${\bf r}_c \approx {\bf r}$.

\begin{summary}
 We summarize that a sum of arbitrarily located potentials in a box can 
 be calculated by a shift-and-windowing tensor operation applied to the low-rank  
 canonical/Tucker representations for the ''reference`` tensor. 
 Usually in electronic structure calculations the $\varepsilon$-rank
 of the resultant tensor sum can be reduced to the quasi-optimal level
 of the same order as the rank of a single ''reference`` tensor.
 \end{summary}

The grid-based representation of a sum of electrostatic
potentials given by $v_c(x)$ in the form of a tensor in the canonical or Tucker format enables
its easy projection to some separable basis set,
like GTO-type atomic orbital basis, polynomials or plane waves.

The following example illustrates that calculation of the 
Galerkin matrix in the Tucker tensor format (cf. \cite{KhorVBAndrae:11,VKH_solver:13} for the case 
of canonical representations) is reduced to a combination of 1D Hadamard and 
scalar products \cite{KhKh3:08}. 
%
Suppose, for simplicity, that the basis set is represented by rank-$1$ canonical tensors, $rank({\bf G}_\mu)=1$, 
representing the basis set, i.e.  
$
{\bf G}_\mu = {\bf g}_\mu^{(1)}\otimes {\bf g}_\mu^{(2)} \otimes {\bf g}_\mu^{(3)}
\in \mathbb{R}^{n \times n \times n},
$
with the canonical vectors ${\bf g}_\mu^{(\ell)}\in\mathbb{R}^{n} $, associated with mode $\ell=1,2,3$,
and $\mu=1,\ldots, N_b$, where $N_b$ is the number of basis functions (vectors).

Suppose that a sum of potentials in a box, $v_c(x)$, given by (\ref{eqn:V_c_GenerP}),
is considered as a multiplicative potential in certain operator (say, the 
Hartree-Fock/Kohn-Sham Hamiltonian).
Given the Tucker tensor approximation to $v_c(x)$ in form (\ref{eqn:WindowTuck_sum}),
with the optimized rank parameters ${\bf r}_c=(r_c,r_c,r_c)$, then its projection onto 
the given basis set is represented by the Galerkin matrix, 
$V_c=\{{v}_{km}\}\in \mathbb{R}^{N_b\times N_b}$,  
whose entries are calculated (approximated) by the simple tensor operations, 
\begin{equation} \label{eqn:nuc_pot}
 {v}_{km}=  \int_{\mathbb{R}^3} v_c(x) {g}_k(x) {g}_m(x) dx \approx 
 \langle {\bf G}_k \odot {\bf G}_m ,   {\bf T}_{{\bf r}_c}\rangle, 
\quad 1\leq k, m \leq N_b,
\end{equation}
where
\[
 {\bf G}_k \odot {\bf G}_m :=
 (  {\bf g}_k^{(1)} \odot  {\bf g}_m^{(1)}  ) \otimes ({\bf g}_k^{(2)} 
\odot  {\bf g}_m^{(2)} ) \otimes  (  {\bf g}_k^{(3)} \odot  {\bf g}_m^{(3)}  ) 
\]
denotes the Hadamard (entrywise) product of rank-$1$ tensors. 
The expression (\ref{eqn:nuc_pot}) can be calculated in terms of 1D Hadamard and 
scalar products with linear complexity $O(n)$. 
 

Similar to the case of Galerkin projection onto the well separable basis set, 
many other tensor operations on the canonical/Tucker 
representations of ${\bf V}_{c}$ can be calculated with the linear cost $O(n)$.

Finally, we notice that the approximation error $\varepsilon >0$ 
caused by a separable representation of the nuclear potential 
is controlled by the rank parameter $r_{c}= rank({\bf T}_{{\bf r}_c})\approx C\, r $,
where $C$ mildly depends on the number of nuclei $M_0$ in a system. 
The exponential convergence of the canonical/Tucker approximation in the rank parameters  allows us 
the optimal choice $r_c =O(|\log \varepsilon |)$ adjusting the complexity 
bound $O(|\log \varepsilon | \, n)$, almost independent on $M_0$. 

\subsection{Assembled lattice sums in a box by using the Tucker format }
\label{ssec:Can2Tuck_lattice}

In this paragraph, we introduce the efficient scheme for fast agglomerated summation
on  a lattice in a box in the Tucker tensor format applied to rather general interaction potentials. 

Given the potential sum $v_c$ in the reference unit cell $\Omega_0=[-b/2,b/2]^3$, of size 
$b\times b \times b$,  we consider an interaction potential in a bounded box
$$
\Omega_L =B_1\times B_2 \times B_3,\quad \mbox{with}\quad 
B_\ell = b/2[- L_\ell ,L_\ell], \; \ell=1,2,3,
$$ 
consisting of a union of  $L_1 \times L_2 \times L_3$ unit cells $\Omega_{\bf k}$,
obtained  by a shift of $\Omega_0$ along the lattice vector $b {\bf k}$, where
${\bf k}=(k_1,k_2,k_3)\in \mathbb{Z}^3$, such that 
$k_\ell \in {\cal K}:={\cal K}_{-}\cup {\cal K}_{+}$ 
for $\ell=1,2,3$ with ${\cal K}_{-}:=\{-1,...,-\frac{L_\ell}{2}\}$ and 
${\cal K}_{+}:=\{0,1,...,\frac{L_\ell}{2}-1\}$.
In the following, for ease of exposition, we  consider a lattice of equal sizes $ L_1= L_2=L_3=L =2L_0$.
By the construction $b=n h$, 
where $h >0$ is the mesh-size that is the same for all spacial variables.
Figure \ref{fig:cube_box} illustrates an example of a 3D lattice structure in a box.

%


\begin{figure}[htb]
\centering
\begin{tikzpicture}[scale=0.36]
 
\draw (2,2) rectangle (12,12);
\path[draw] (2,12) -- (6,15); 
\path[draw] (12,12) -- (16,15); 

\draw[style=dashed] (6,5) rectangle (16,15);
\path[draw][style=dashed] (2,2) -- (6,5); 
\path[draw] (12,2) -- (16,5); 

\foreach \x in {8,9,10,11,12,13} {
\foreach \y in {7.2,8.2,9.2,10.2,11.2,12.2} {
\shade[shading = ball, ball color = blue] (\x,\y) circle (.2);
}}

\foreach \x in {8,9,10,11,12,13} {
\path[draw] [color = blue](\x,7.2) -- (\x,12.2); }

\foreach \y in {7.2,8.2,9.2,10.2,11.2,12.2} {
\path[draw] [color = blue](8,\y) -- (13,\y); }

\foreach \x in {7.4,8.4,9.4,10.4,11.4,12.4} {
\foreach \y in {6.7,7.7,8.7,9.7,10.7,11.7} {
\shade[shading = ball, ball color = blue] (\x,\y) circle (.2);
}}

\foreach \x in {7.4,8.4,9.4,10.4,11.4,12.4} {
\path[draw] [color = blue](\x,6.9) -- (\x,11.9); }

\foreach \y in {6.7,7.7,8.7,9.7,10.7,11.78} {
\path[draw] [color = blue](7.4,\y) -- (12.4,\y); }

\foreach \x in {6.8,7.8,8.8,9.8,10.8,11.8} {
\foreach \y in {6.2,7.2,8.2,9.2,10.2,11.2} {
\shade[shading = ball, ball color = blue] (\x,\y) circle (.2);
}}

\foreach \x in {6.8,7.8,8.8,9.8,10.8,11.8} {
\path[draw] [color = blue](\x,6.2) -- (\x,11.2); }

\foreach \y in {6.2,7.2,8.2,9.2,10.2,11.2} {
\path[draw] [color = blue](6.8,\y) -- (11.8,\y); }

\foreach \x in {6.2,7.2,8.2,9.2,10.2,11.2} {
\foreach \y in {5.7,6.7,7.7,8.7,9.7,10.7} {
\shade[shading = ball, ball color = blue] (\x,\y) circle (.2);
}}

\foreach \x in {6.2,7.2,8.2,9.2,10.2,11.2} {
\path[draw] [color = blue](\x,5.7) -- (\x,10.7); }

\foreach \y in {5.7,6.7,7.7,8.7,9.7,10.7} {
\path[draw] [color = blue](6.2,\y) -- (11.2,\y); }
\end{tikzpicture}
\caption{Rectangular  $ 6\times 6\times 4$ lattice in a box.}
\label{fig:cube_box}
\end{figure}
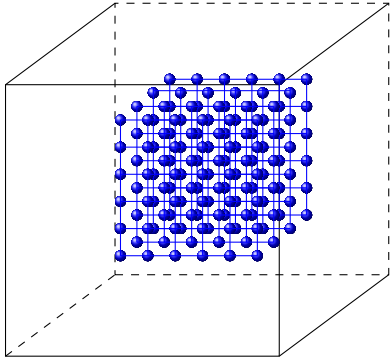

The potential $v_{c_L}(x)$, for $x\in \Omega_{L}$ 
is obtained by summation over all unit cells $\Omega_{\bf k}$ in $\Omega_L$,
\begin{equation}\label{eqn:EwaldSumE}
v_{c_L}(x)=  \sum_{\nu=1}^{M_0} Z_\nu \sum\limits_{k_1,k_2,k_3\in {\cal K}} 
p(\|{x} - a_\nu - b {\bf k} \|), \quad x\in \Omega_{L}. 
\end{equation}
Note that conventionally this calculation is performed at each of $L^3$ unit cells 
$\Omega_{\bf k}\subset \Omega_L$, ${\bf k}\in {\cal K}^3$, on the rectangular lattice, which presupposes substantial 
numerical costs at least of the order of $O(L^3)$. 
The presented approach applies not only to the complete rectangular $L\times L\times L$ lattice,
but remains efficient in the case of defected  lattices and for more complicated symmetries. 
It allows to essentially reduce these costs to linear scaling in $L$.

Let $\Omega_{N_L}$ be the $N_L\times N_L\times N_L$ uniform grid on $\Omega_L$ with the 
same mesh-size $h$
as above, and introduce the corresponding space of piecewise constant basis functions 
of the dimension $N_L^3$. In this construction we have $N_L = Ln$.
In the case of canonical sums, we simply follow \cite{VeBoKh:Ewald:14},
and employ, similar to (\ref{master_pot}),  the rank-$R$ ''reference`` tensor defined 
on the larger auxiliary box $\widetilde{\Omega}_{L}$ by scaling ${\Omega}_{L}$ with a factor of $2$,
\begin{equation}\label{eqn:CanMastSumNL}
\widetilde{\bf P}_{{L},R}= \sum\limits_{q=1}^{R} \widetilde{\bf p}^{(1)}_q 
\otimes \widetilde{\bf p}^{(2)}_q \otimes \widetilde{\bf p}^{(3)}_q 
\in \mathbb{R}^{2 N_L\times 2N_L\times 2N_L}.
\end{equation}
Along the same line as in (\ref{eqn:TuckRef_sum}), we introduce the rank-${\bf r}$ ''reference`` Tucker tensor 
$\widetilde{\bf T}_{{L},{\bf r}}\in \mathbb{R}^{2 N_L\times 2 N_L \times 2 N_L}$ 
defined on the auxiliary domain $\widetilde{\Omega}_{L}$.

The next theorem generalizes Theorem 3.1 in \cite{VeBoKh:Ewald:14} to the case
of general function $p(\|x\|)$ in (\ref{eqn:EwaldSumE}) as well as to the case
of Tucker tensor decompositions.
It proves the storage and numerical costs for the lattice sum  of 
single potentials (i.e. corresponding to the choice $M_0=1$, and $a_1 =0$ in (\ref{eqn:EwaldSumE})),
each represented by a rank-$R$ canonical or rank-${\bf r}$ Tucker tensors. 
In what following the windowing operator 
${\cal W} = {\cal W}_{({\bf k})}={\cal W}_{(k_1)}\otimes {\cal W}_{(k_2)}\otimes {\cal W}_{(k_3)}$ 
specifies a shift by the lattice vector $b{\bf k}$.
\begin{theorem}\label{thm:sumCaseE}
(A) Given the rank-$R$ canonical "reference" tensor (\ref{eqn:CanMastSumNL}) approximating
the potential $p(\|x\|)$.
The projected tensor of the interaction potential, ${\bf V}_{c_L}$,   
representing the full lattice sum over $L^3$ cells can be presented by the rank-$R$
canonical tensor ${\bf P}_{c_L}$,
\begin{equation}\label{eqn:EwaldTensorGl}
{\bf P}_{c_L}= 
\sum\limits_{q=1}^{R}
(\sum\limits_{k_1\in {\cal K}} {\cal W}_{({k_1})} \widetilde{\bf p}^{(1)}_{q}) \otimes 
(\sum\limits_{k_2\in {\cal K}} {\cal W}_{({k_2})} \widetilde{\bf p}^{(2)}_{q}) \otimes 
(\sum\limits_{k_3\in {\cal K}} {\cal W}_{({k_3})} \widetilde{\bf p}^{(3)}_{q}).
\end{equation}
The numerical cost and storage size are estimated by $O(R L N_L)$ and $O(R N_L)$,
respectively, where $N_L=n L$ is the univariate grid size. 

(B) Given the rank-${\bf r}$ ''reference`` Tucker tensor 
$\widetilde{\bf T}_{{L},{\bf r}}\in \mathbb{R}^{2 N_L\times 2 N_L \times 2 N_L}$, 
see (\ref{eqn:TuckRef_sum}), approximating the potential function $p(\|x\|)$.
The rank-${\bf r}$ Tucker approximation of a lattice-sum tensor ${\bf V}_{c_L}$ 
can be computed in the form
\begin{equation} \label{eqn:Tuck_LatticeSum}
\begin{split}
 {\bf T}_{c_L} & = 
 \sum\limits_{{\bf m}= {\bf 1}}^{\bf r} b_{\bf m}
(\sum\limits_{k_1\in {\cal K}} {\cal W}_{({k_1})} \widetilde{\bf t}^{(1)}_{m_1}) \otimes 
(\sum\limits_{k_2\in {\cal K}} {\cal W}_{({k_2})} \widetilde{\bf t}^{(2)}_{m_2}) \otimes 
(\sum\limits_{k_3\in {\cal K}} {\cal W}_{({k_3})} \widetilde{\bf t}^{(3)}_{m_3}). 
\end{split}
\end{equation}
The numerical cost and storage size are estimated by $O(3 r L N_L)$ and $O(3r N_L)$, respectively.
\end{theorem}
\begin{proof}
Conventionally, we fix the index $\nu=1$ in (\ref{eqn:EwaldSumE}), set $a_\nu =0$ and $Z_1=1$, 
and consider only the second sum  defined on the 
complete domain $\Omega_L$,
\begin{equation}\label{eqn:EwaldSumGl}
{v}_{c_L} (x)=   \sum\limits_{k_1,k_2,k_3 \in {\cal K}}
p(\|{x} - b {\bf k} \|), \quad x\in  \Omega_L.
\end{equation}
Then the projected tensor representation of ${v}_{c_L} (x)$ takes the form 
\[
  {\bf P}_{c_L}= \sum\limits_{k_1,k_2,k_3 \in {\cal K}}  {\cal W}_{\nu({\bf k})}  
 \widetilde{\bf P}_{{L},R}= \sum\limits_{k_1,k_2,k_3 \in {\cal K}} \sum\limits_{q=1}^{R}
{\cal W}_{ ({\bf k})}( \widetilde{\bf p}^{(1)}_{q} \otimes \widetilde{\bf p}^{(2)}_{q} 
\otimes \widetilde{\bf p}^{(3)}_{q})
\in \mathbb{R}^{N_L\times N_L  \times N_L},
\]
where the 3D shift vector is defined by  ${\bf k}=(k_1,k_2,k_3)\in \mathbb{Z}^{L\times L\times L}$. 
Taking into account the rank-$1$ separable representation 
of the $\Omega_L$-windowing operator (tracing onto $N_L\times N_L\times N_L$ window),
$$
{\cal W}_{({\bf k})}={\cal W}_{ (k_1)}^{(1)}\otimes {\cal W}_{ (k_2)}^{(2)}
\otimes {\cal W}_{ (k_3)}^{(3)},
$$ 
we rewrite the above summation as 
\begin{equation}\label{eqn:fullsum}
  {\bf P}_{c_L}=  \sum\limits_{q=1}^{R}\sum\limits_{k_1,k_2,k_3 \in {\cal K}}
{\cal W}_{ ({k_1})} \widetilde{\bf p}^{(1)}_{q} \otimes {\cal W}_{ ({k_2})} \widetilde{\bf p}^{(2)}_{q} 
\otimes {\cal W}_{ ({k_3})} \widetilde{\bf p}^{(3)}_{q}.
\end{equation} 
To reduce the large sum over the full 3D lattice, we use the following property of a sum
of canonical tensors, ${\bf C}= {\bf A} + {\bf B}$, with equal ranks $R$ 
and with two coinciding factor matrices, say for $\ell=1,2$: 
the concatenation in the remaining mode $\ell=3$  can be reduced to a 
pointwise summation of the respective canonical vectors,
\begin{equation}
\label{eqn:conc2sum}
C^{(3)} =[{\bf a}_1^{(3)}+ {\bf b}_1^{(3)}, \ldots ,
{\bf a}_{R}^{(3)}+ {\bf b}_{R}^{(3)}],
\end{equation} 
while the first two mode vectors remain unchanged, 
$C^{(1)}=A^{(1)}= B^{(1)}$, $C^{(2)}=A^{(2)}=B^{(2)}$.
This preserves the same rank parameter $R$ for the resulting sum.
Notice that for each fixed $q$ the inner sum in (\ref{eqn:fullsum}) satisfies the above
property. Repeatedly applying this property to a large number of canonical tensors, 
the  3D-sum (\ref{eqn:fullsum})   is reduced to  a rank-$R$ tensor obtained 
by 1D summations only,
\[
\begin{split}
 {\bf P}_{c_L} & = \sum\limits_{q=1}^{R}
(\sum\limits_{k_1 \in {\cal K}} {\cal W}_{ ({k_1})} \widetilde{\bf p}^{(1)}_{q}) \otimes
 (\sum\limits_{k_2,k_3 \in {\cal K}} {\cal W}_{ ({k_2})} \widetilde{\bf p}^{(2)}_{q} 
\otimes {\cal W}_{ ({k_3})} \widetilde{\bf p}^{(3)}_{q})\\
 &= \sum\limits_{q=1}^{R} 
 (\sum\limits_{k_1 \in {\cal K}} {\cal W}_{ ({k_1})} \widetilde{\bf p}^{(1)}_{q}) \otimes 
(\sum\limits_{k_2 \in {\cal K}} {\cal W}_{ ({k_2})} \widetilde{\bf p}^{(2)}_{q}) \otimes 
(\sum\limits_{k_3 \in {\cal K}} {\cal W}_{ ({k_3})} \widetilde{\bf p}^{(3)}_{q}).
\end{split}
\]
The numerical cost are estimated by using the standard properties of canonical tensors.

 In the case of Tucker representation we apply the similar argument to obtain 
\begin{equation*} \label{eqn:Tuck_LatticeSumF}
\begin{split}
 {\bf T}_{c_L} & = 
\sum\limits_{k_1,k_2,k_3 \in {\cal K}} {\cal W}_{({\bf k})} \widetilde{\bf T}_{L,{\bf r}} \\
               & =
 \sum\limits_{{\bf m}= {\bf 1}}^{\bf r} b_{\bf m}
(\sum\limits_{k_1 \in {\cal K}} {\cal W}_{({k_1})} \widetilde{\bf t}^{(1)}_{m_1}) \otimes 
(\sum\limits_{k_2 \in {\cal K}} {\cal W}_{({k_2})} \widetilde{\bf t}^{(2)}_{m_2}) \otimes 
(\sum\limits_{k_3 \in {\cal K}} {\cal W}_{({k_3})} \widetilde{\bf t}^{(3)}_{m_3}). 
\end{split}
\end{equation*}
Simple complexity estimates complete the proof.
\end{proof}

Figure \ref{fig:TuckCan_assembl} illustrates the shape of several Tucker vectors obtained 
by assembling vectors $\widetilde{\bf t}^{(1)}_{m_1}$ along $x_1$-axis. 
It can be seen that assembled Tucker vectors accumulate simultaneously 
the contributions of all single potentials involved in the total sum.
Note that the assembled Tucker
vectors do not preserve the initial orthogonality of directional vectors
$\{\widetilde{\bf t}^{(\ell)}_{m_\ell}\}$. In this case the simple 
Gram-Schmidt orthogonalization can be applied.
\begin{figure}[htbp]
\centering
\includegraphics[width=6.0cm]{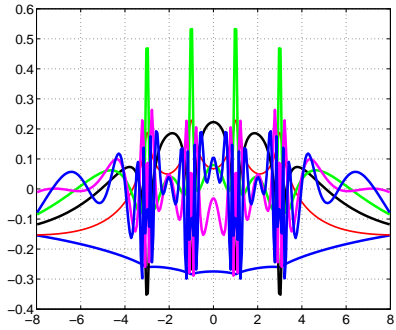}
\caption{Assembled Tucker vectors by using $\widetilde{\bf t}^{(1)}_{m_1}$ 
along the $x_1$-axis, for a sum over lattice $4\times 4\times 1$.}
\label{fig:TuckCan_assembl}  
\end{figure}
\begin{remark}
In the general case $M_0 >1$, the weighted summation over $M_0$ charges leads 
to the rank-$R_c$ canonical tensor representation on the ''reference'' domain 
$\widetilde{\Omega}_L$, which can be used to obtain the rank-$R_c$ 
representation of a sum in the whole $L\times L\times L$ lattice 
\begin{equation}\label{eqn:EwaldTensorM0}
{\bf P}_{c_L}= 
\sum\limits_{q=1}^{R_c}
(\sum\limits_{k_1\in {\cal K}} {\cal W}_{({k_1})} \widetilde{\bf p}^{(1)}_{q}) \otimes 
(\sum\limits_{k_2\in {\cal K}} {\cal W}_{({k_2})} \widetilde{\bf p}^{(2)}_{q}) \otimes 
(\sum\limits_{k_3\in {\cal K}} {\cal W}_{({k_3})} \widetilde{\bf p}^{(3)}_{q}).
\end{equation}
Likewise, the rank-${\bf r}_c$ Tucker approximation of a tensor ${\bf V}_{c_L}$ can be computed in the form
\begin{equation} \label{eqn:Tuck_LatticeSumM0}
\begin{split}
{\bf T}_{c_L} & = 
 \sum\limits_{{\bf m}= {\bf 1}}^{{\bf r}_0} b_{\bf m}
(\sum\limits_{k_1\in {\cal K}} {\cal W}_{({k_1})} \widetilde{\bf t}^{(1)}_{m_1}) \otimes 
(\sum\limits_{k_2\in {\cal K}} {\cal W}_{({k_2})} \widetilde{\bf t}^{(2)}_{m_2}) \otimes 
(\sum\limits_{k_3\in {\cal K}} {\cal W}_{({k_3})} \widetilde{\bf t}^{(3)}_{m_3}). 
\end{split}
\end{equation}
\end{remark}
The next remark generalizes the basic construction to the case of non-uniformly spaced
rectangular lattices.
\begin{remark}
The previous construction applies to the uniformly spaced positions of charges.
However, the agglomerated tensor summation method in both canonical and Tucker formats
applies with slight modification of the windowing operator to a
non-equidistant $L_1 \times L_2 \times L_3 $  tensor lattice. Such lattice sums could 
not be treated by the traditional Ewald summation methods based on the FFT transform.
\end{remark}
\begin{figure}[htbp]
\centering
\includegraphics[width=7.0cm]{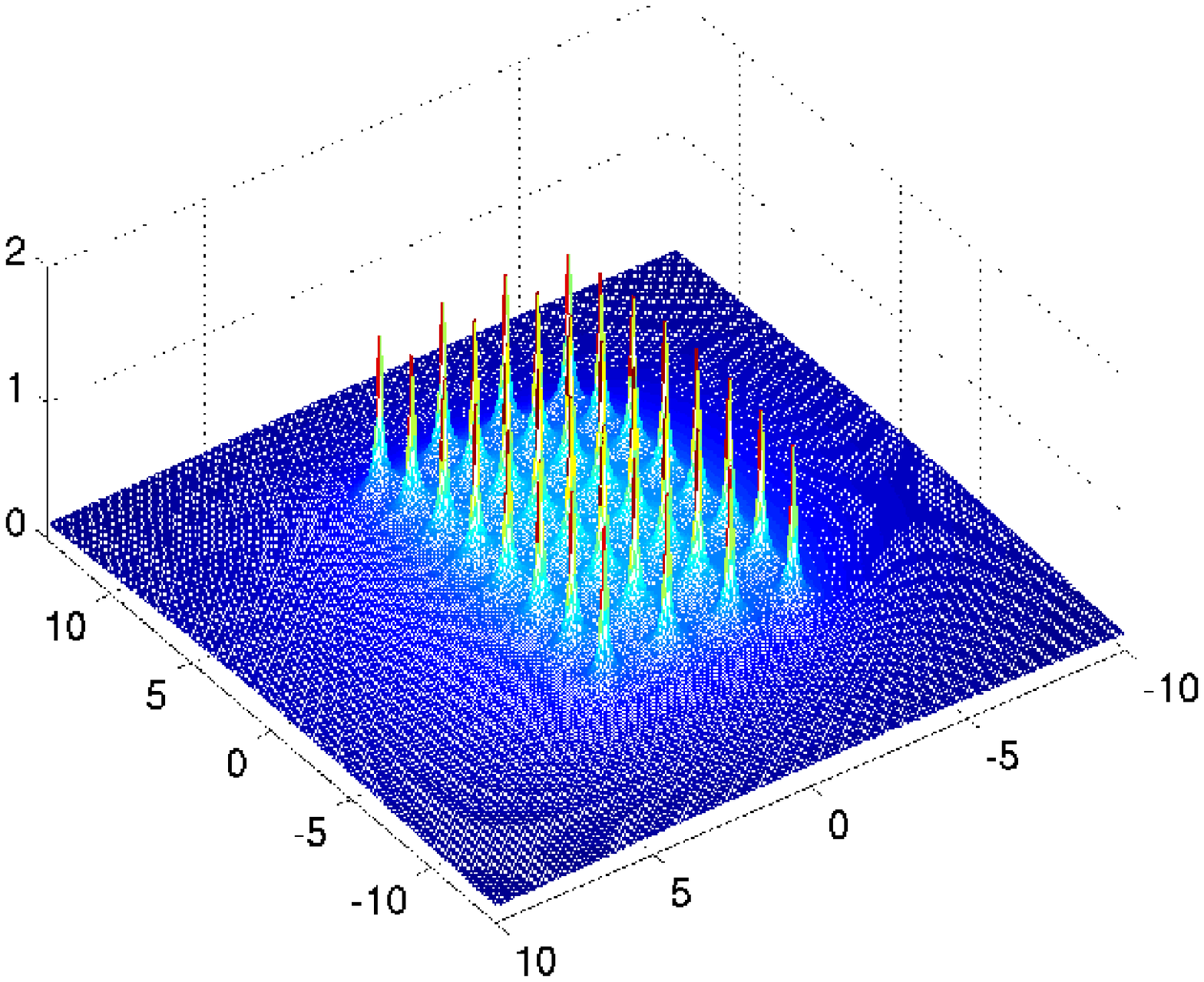}
\includegraphics[width=7.0cm]{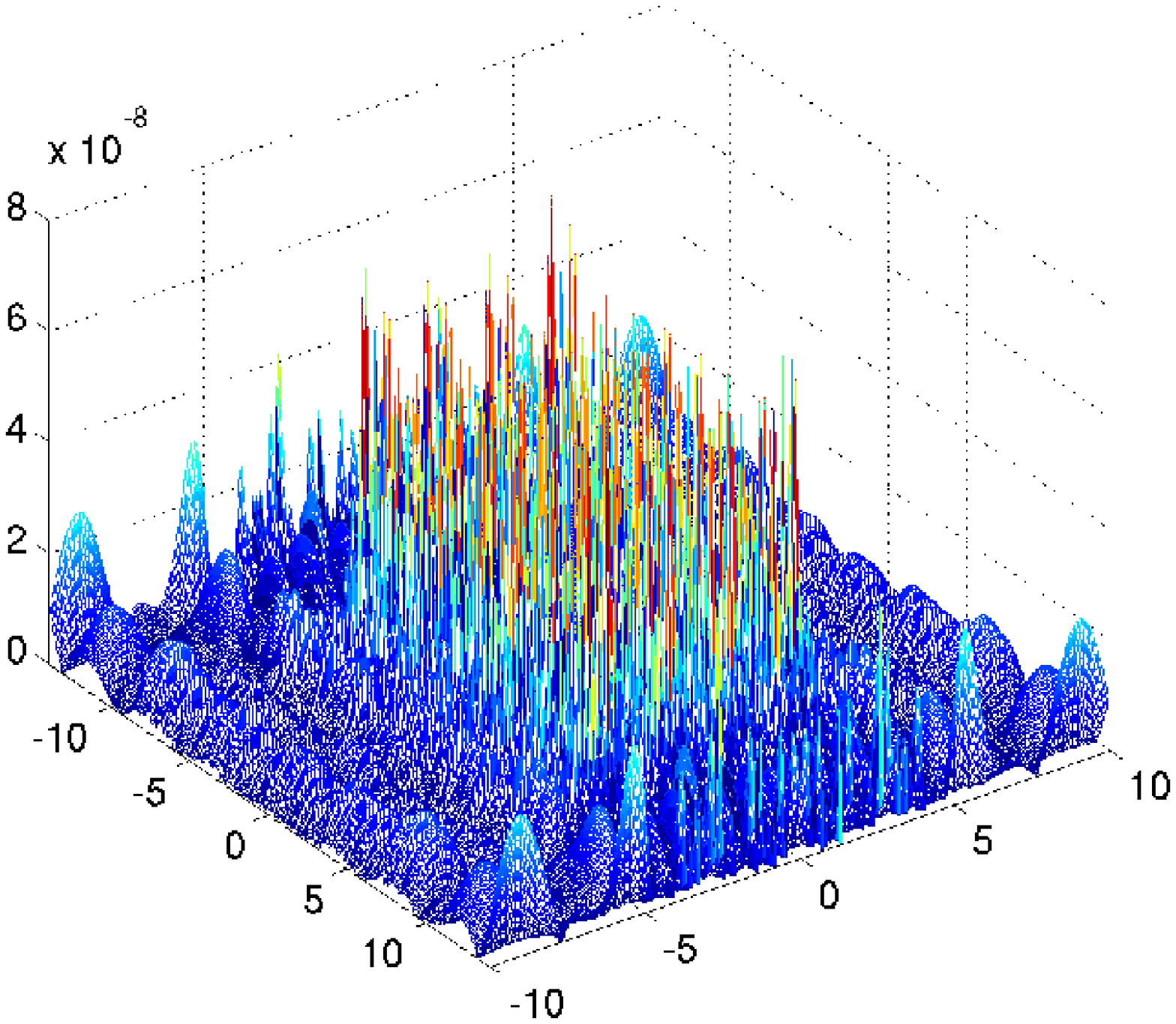}
\caption{Left: Sum of Newton potentials on a $8 \times 4\times 1$ lattice  
generated in a volume with the 3D grid of size $14336 \times 10240 \times 7168$. 
Right: the absolute approximation error (about $8\cdot 10^{-8}$) in the Tucker format.
}
\label{fig:Newt_pot_841}  
\end{figure}
\begin{table}[tbh]
\begin{center}%
\begin{tabular}
[c]{|r|r|r|r|r|}%
\hline\hline
$L^3$         & $4096$  & $32768$  & $262144$ & $2097152$ \\
\hline
Time    & $1.8$  & $0.8$  & $3.1$  &  $15.8$ \\
\hline  
$N_L^3$         & $5632^3$  & $9728^3$  & $17920^3$ & $34304^3$ \\
\hline \hline
 \end{tabular}
\caption{Time (sec.) vs. the total number of potentials $L^3$ for the assembled Tucker 
calculation of the lattice sum ${\bf T}_{c_L}$. 
Mesh size (for all grids) is $h=0.0034$ $\AA{}$. }
\label{Table_timesL}
\end{center}
\end{table}
Both the Tucker and canonical tensor representations (\ref{eqn:Tuck_LatticeSum})
and (\ref{eqn:EwaldTensorGl}) reduce dramatically the numerical costs and storage consumptions. 
Table \ref{Table_timesL} illustrates  complexity scaling $O(N_L L)$ for  
computation of  $L\times L\times L$ lattice sum in the Tucker format,
where the grid-size is given by $N_L\times N_L\times N_L$ with $N_L=n\, L$.
These results confirm our theoretical estimates.

Figure \ref{fig:Newt_pot_841} shows the sum of Newton kernels on a lattice 
$8 \times 4\times 1 $ and the respective Tucker summation error achieved
on the large 3D representation grid with the rank ${\bf r}=(16,16,16)$ Tucker tensor.
The spacial mesh size is about $0.002$ atomic units ($0.001$ \AA{}).

\begin{figure}[htbp]
\centering
\includegraphics[width=5.3cm]{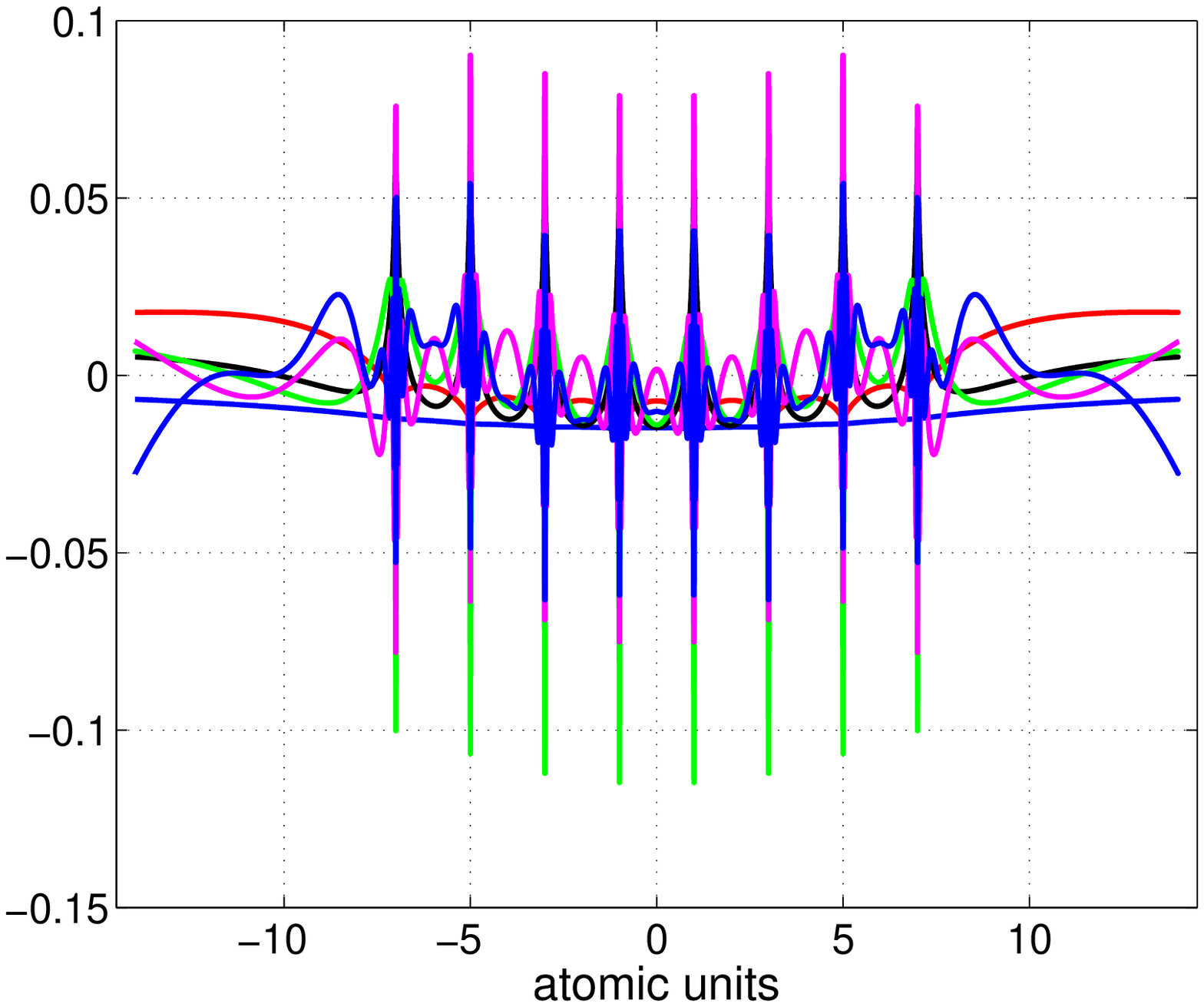}
\includegraphics[width=5.3cm]{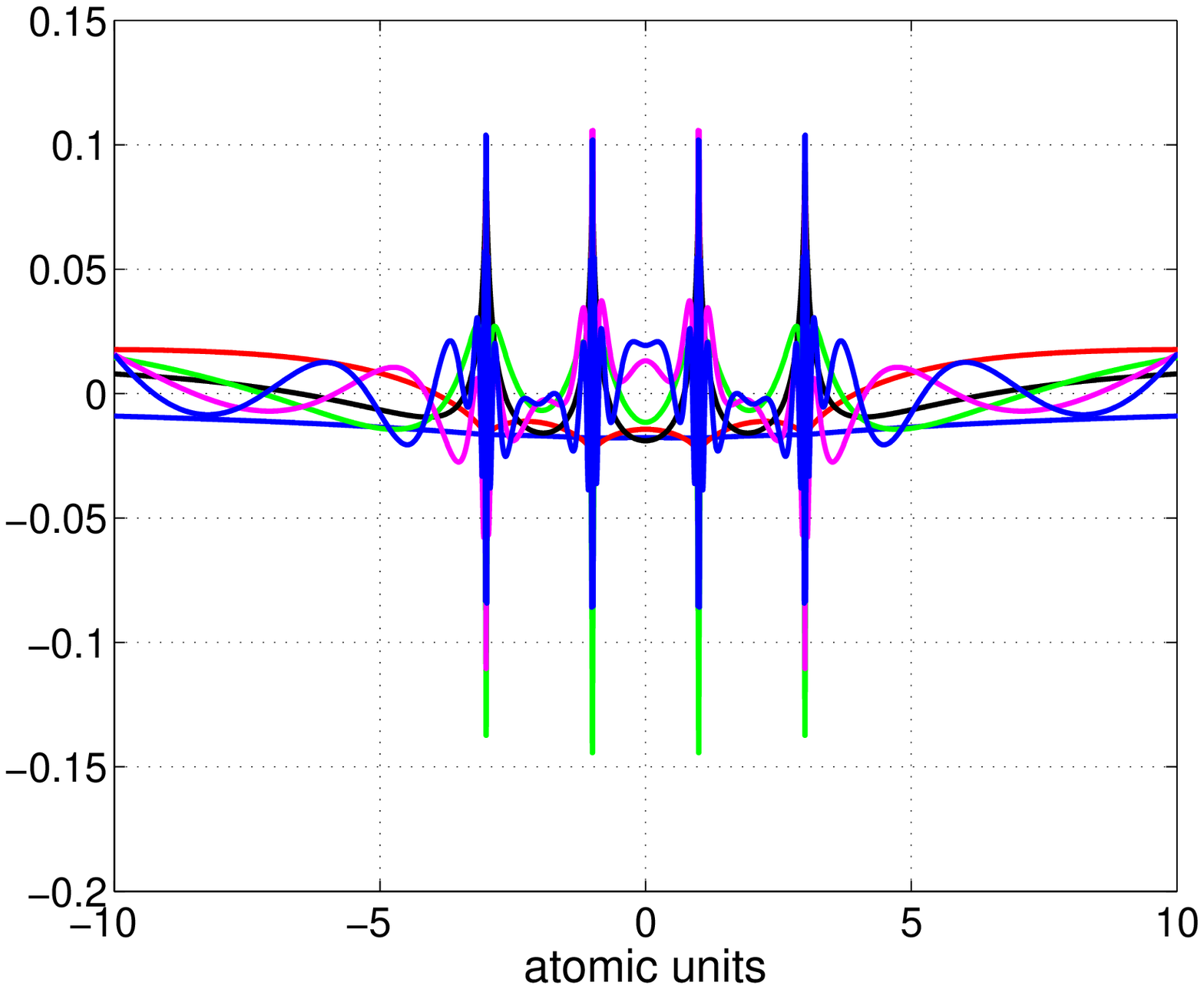}
\includegraphics[width=5.3cm]{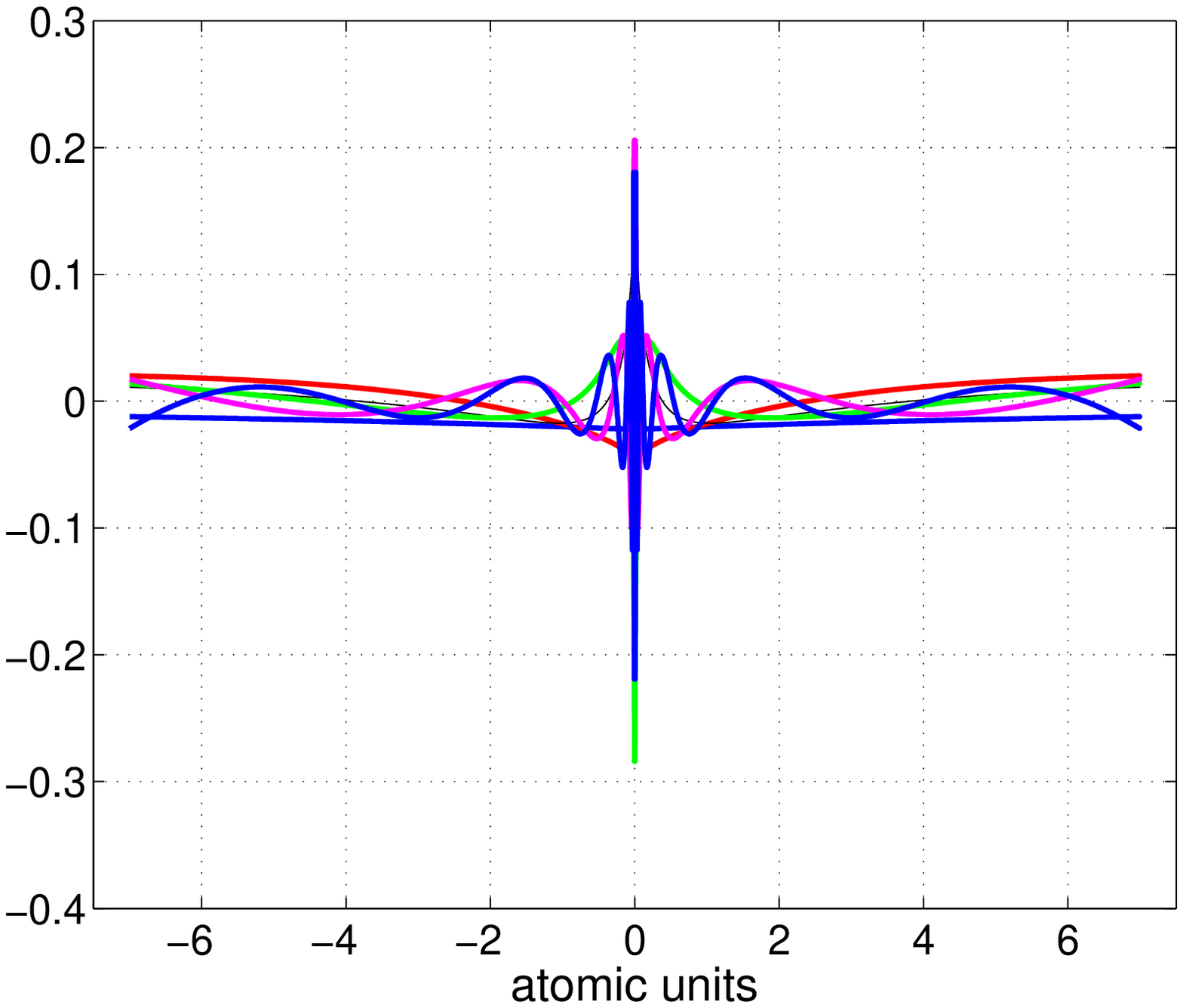}
\caption{Several mode vectors from the C2T approximation 
visualized along $x,y$- and $z$-axis  on a $8 \times 4\times 1$ lattice.}
\label{fig:Tuck_841}  
\end{figure}
Figure \ref{fig:Tuck_841} represents the Tucker vectors obtained from the
canonical-to-Tucker (C2T) approximation of the assembled canonical 
tensor sum of potentials on a $8 \times 4 \times1$ lattice.
In this case the Tucker vectors are orthogonal.


\section{Potential sums on defected lattice} \label{sec:Vac_lattice}

\subsection{Problem setting}
\label{ssec:Vac_latticeProbl}

For lattice sums on the perfect rectangular geometries the resultant canonical and Tucker tensors are proven to 
inherit exactly the same rank parameters as those for the single ''reference`` tensor.
In the case of lattices with defects, say, vacancies or impurities, 
the tensor rank of agglomerated sums in both canonical and Tucker formats may increase dramatically.
In such cases the rank reduction procedure is required.
  
In this section, we analyze the assembled summation of Tucker/canonical tensors 
on the defected lattices  in the  algebraic framework as follows.
Denote the perturbed Tucker tensor by $\widehat{\bf U}$. Let us introduce a set of ${\bf k}$-indices on the  lattice, 
${\cal S}=:\{{\bf k}_1,...,{\bf k}_S\}$, where the unperturbed  Tucker tensor ${\bf U}_{0}:={\bf T}_{c_L}$
initially given by summation over the full rectangular lattice (\ref{eqn:Tuck_LatticeSum}) 
is perturbed (defected) at positions associated with ${\bf k}\in {\cal S}$ by the Tucker tensor 
${\bf U}_{\bf k}={\bf U}_{s}$ ($s=1,...,S$), given by,
\begin{equation} \label{eqn:Tuck_LatticeKi}
\begin{split}
 {\bf U}_{s} & = 
 \sum\limits_{{\bf m}= {\bf 1}}^{{\bf r}_{s}} b_{s,{\bf m}}
 {\bf u}^{(1)}_{s,m_1} \otimes 
 {\bf u}^{(2)}_{s,m_2} \otimes 
 {\bf u}^{(3)}_{s,m_3}, \quad s=1,...,S. 
\end{split}
\end{equation}
Without loss of generality, all Tucker tensors ${\bf U}_{s}$, ($s=0,1,...,S$),
can be assumed orthogonal.

Now the perturbed Tucker tensor $\widehat{\bf U}$ is obtained from the non-perturbed one, 
${\bf U}_{0}$, by adding a sum  of all defects ${\bf U}_{\bf k}$, ${\bf k}\in {\cal S}$,
\begin{equation} \label{eqn:Tuck_LatticeSum_full}
 {\bf U}_{0} \mapsto \widehat{\bf U}={\bf U}_{0} + \sum\limits_{s=1}^S{\bf U}_{s},
\end{equation}
which implies the upper rank estimates for best Tucker approximation of $\widehat{\bf U}$,
$$
\widehat{r}_\ell \leq r_{0,\ell} + \sum\limits_{s=1}^S r_{s,\ell},\quad \mbox{for} \quad \ell=1,2,3.
$$
If the number of perturbed cells, $S$, is large enough, the numerical
computations with the Tucker tensor of rank $\widehat{r}_\ell$ becomes prohibitive
and the rank reduction procedure is required.

In the case of lattice sum in the Tucker format, we propose the generalization to the RHOSVD 
algorithm, that applies directly to a large sum of Tucker tensors. 
In this way the initial RHOSVD algorithm in \cite{KhKh3:08} can be viewed  
as the special case of generalized RHOSVD scheme now applied to a sum of rank-one Tucker tensors. 
The stability of the new rank reduction method can be proven
under mild assumptions on the ''weak orthogonality`` of the Tucker tensors representing defects in the 
lattice sum.
The numerical complexity of the generalized RHOSVD algorithm scales only linearly
in the number of vacancies.


We use the similar notation to describe the summation of canonical tensors on defected lattices.
The non-perturbed canonical tensor ${\bf P}_{0}:={\bf P}_{c_L}$ given by (\ref{eqn:EwaldTensorGl}) 
is substituted by a sum of canonical  tensors representing the expected
perturbations,
\begin{equation} \label{eqn:Can_LatticeSum_full}
 {\bf P}_{0} \mapsto \widehat{\bf P}={\bf P}_{0} + \sum\limits_{s=1}^S{\bf P}_{s}
\end{equation}
with the upper rank estimate for best canonical approximation of 
the perturbed canonical tensor $\widehat{\bf P}$,
\begin{equation} \label{eqn:RankCan_def}
\widehat{r} \leq r_{0} + \sum\limits_{s=1}^S r_{s}. 
\end{equation}
Again, the rank reduction procedure is normally required.

\subsection{Defected lattice sum of canonical tensors}
\label{ssec:Can2Tuck_rev}

We consider a sum of canonical tensors on a lattice with defects located at $S$ sources.
In accordance with (\ref{eqn:Can_LatticeSum_full}) - (\ref{eqn:RankCan_def}), the canonical rank of 
the resultant tensor  may increase at a factor of $S$.
The effective rank of the perturbed sum may be reduced by using the RHOSVD approximation via
Can $\mapsto$ Tuck $\mapsto$ Can algorithm, proposed in \cite{KhKh3:08}. 
This approach basically provides the compressed tensor
with the canonical rank quadratically proportional to those of the respective Tucker
approximation to the sum with defects.
For the readers convenience, in Appendix, we recall the error estimate for RHOSVD approximation
to sums of canonical tensors \cite{KhKh3:08}.  

In what follows, we discuss the stability conditions for RHOSVD approximation
and their applicability in the summation on spherically symmetric interaction potentials.
Given a rank parameter $R\in \mathbb{N}$, we denote by  
 \begin{equation}\label{eqn:CP_form}
   {\bf A}={\sum}_{\nu =1}^{R} \xi_{\nu}
   {\bf a}^{(1)}_{\nu}  \otimes \ldots \otimes {\bf a}^{(3)}_{\nu},
 \quad  \xi_{\nu}\in \mathbb{R},
\end{equation}
the canonical tensor with normalized vectors ${\bf a}_{\nu}^{(\ell)}\in \mathbb{R}^{n_\ell} $ ($\ell=1,...,3$)
that is defined by the side-matrices 
$A^{(\ell)}=\left[{\bf a}^{(\ell)}_{1}...{\bf a}^{(\ell)}_{R}\right], \; 
A^{(\ell)}\in \mathbb{R}^{n \times R}$, obtained by concatenation of 
the corresponding canonical vectors in (\ref{eqn:CP_form}).
The minimal parameter $R$ in (\ref{eqn:CP_form}) is called
the rank (or canonical rank) of a tensor. 
The representation (\ref{eqn:CP_form}) can be written as the rank-$(R,R,R)$ Tucker tensor 
by introducing the diagonal Tucker core tensor
$\boldsymbol{\xi}:=\mbox{diag} \{ \xi_1,...,\xi_R\}\in \mathbb{R}^{R\times R \times R}$ such that
$\xi_{\nu_1,\nu_2,\nu_3}=0$ except when $\nu_1=...=\nu_3$ with 
$\xi_{\nu,...,\nu}=\xi_{\nu}$ ($\nu=1,...,R$), 
\begin{equation}\label{eqn:CP_form_ContrpTuck}
 {\bf A}=\boldsymbol{\xi} \times_1 {A}^{(1)}\times_2 {A}^{(2)} \times_d {A}^{(3)}. 
\end{equation}

Given the rank parameter ${\bf r}=(r_1,r_2,r_3)$.
To define the reduced rank-$\bf r$ HOSVD type
Tucker approximation to the tensor in (\ref{eqn:CP_form}),
we set $n_\ell=n$ and suppose for definiteness that $n\leq R$, so that SVD of the
side-matrix $A^{(\ell)}$ is given by
\[
A^{(\ell)}={Z}^{(\ell)} D_\ell {V^{(\ell)}}^T=
\sum\limits_{k=1}^n \sigma_{\ell,k} {\bf z}_k^{(\ell)}\; {{\bf v}_k^{(\ell)}}^T,\quad
{\bf z}_k^{(\ell)}\in \mathbb{R}^{n},\; {\bf v}_k^{(\ell)} \in \mathbb{R}^{R},
\]
with the orthogonal matrices ${ Z}^{(\ell)}=[{\bf z}_1^{(\ell)},...,{\bf z}_n^{(\ell)}]$, and
${V}^{(\ell)}=[{\bf v}_1^{(\ell)},...,{\bf v}_n^{(\ell)}]$, $\ell=1,2,3$.
Given rank parameters $r_1,...,r_\ell <n $, introduce the truncated SVD of the side-matrix
$A^{(\ell)}$, ${Z}_0^{(\ell)} D_{\ell,0} {V_0^{(\ell)}}^T$, ($\ell=1,2,3$),
where
$
D_{\ell,0}=\mbox{diag} \{\sigma_{\ell,1},\sigma_{\ell,2},...,\sigma_{\ell,r_\ell}\}
$
and ${Z}_0^{(\ell)}\in \mathbb{R}^{n\times r_\ell}$,
${V_0}^{(\ell)}\in \mathbb{R}^{R\times r_\ell} $, represent the orthogonal
factors being the respective sub-matrices in the SVD factors of ${A}^{(\ell)}$.
\begin{definition}\label{def:RHOSVD}
(\cite{KhKh3:08}) 
The reduced HOSVD (RHOSVD) approximation of ${\bf A}$, further called ${\bf A}_{({\bf r})}^0$, 
is defined as the rank-${\bf r}$ Tucker tensor obtained by the projection  of ${\bf A}$ 
in the form (\ref{eqn:CP_form_ContrpTuck}) onto the 
orthogonal matrices of the dominating singular vectors in $Z_0^{(\ell)}$, ($\ell=1,2,3$). 
\end{definition}

 The stability of RHOSVD approximation is formulated in the following assertion. 
\begin{lemma}\label{lem:Can_stab}
Let decomposition (\ref{eqn:CP_form}) satisfy the stability condition 
\begin{equation}\label{Eqn:stabCan}
 \sum\limits_{\nu =1}^{R} \xi_{\nu}^2 \leq C \|{\bf A}\|^2,
\end{equation}
then the quasi-optimal RHOSVD approximation is robust in the relative norm 
\[
 \|{\bf A}- {\bf A}_{({\bf r})}^0\| \leq C \|{\bf A}\| \sum\limits_{\ell=1}^3
(\sum\limits_{k=r_\ell +1}^{\min(n,R)} \sigma_{\ell,k}^2)^{1/2},
\]
where $\sigma_{\ell,k}$ ($k=r_\ell +1,...,n$) denote the truncated singular values.
\end{lemma}
\begin{proof}
 The proof is a simple consequence of the general error estimate (\ref{eqn:error_bound_CtoT}).
\end{proof}

The stability condition (\ref{Eqn:stabCan}) is fulfilled, in particular, if 

(a) All canonical vectors in (\ref{eqn:CP_form}) are non-negative 
that is the case for $sinc$-quadrature based approximations to Green's kernels based on
integral transforms (\ref{eqn:LalpSlater}) - (\ref{eqn:LalpPoly}), since $a_k >0$.

(b) The partial orthogonality of the canonical vectors holds, i.e. rank-$1$ tensors 
${\bf a}^{(1)}_{\nu}  \otimes \ldots \otimes {\bf a}^{(d)}_{\nu} $ ($\nu=1,...,R$) 
are mutually orthogonal.
We refer to \cite{Kolda:01} for various definitions of orthogonality for canonical tensors.

\subsection{Summation on defected lattice in the Tucker tensor  format }
\label{ssec:TuckSum2Tuck} 

In the case of Tucker sum (\ref{eqn:Tuck_LatticeSum_full}) we define the agglomerated side 
matrices $\widehat{U}^{(\ell)}$ by concatenation
of the directional side-matrices of individual tensors ${\bf U}_s$, $s=0,1,...,S$,
\begin{equation}\label{eqn:SideMTuckS}
\widehat{U}^{(\ell)}=[{\bf u}^{(\ell)}_{1}...{\bf u}^{(\ell)}_{r_{0,\ell}}, {\bf u}^{(\ell)}_{1}...
{\bf u}^{(\ell)}_{r_{1,\ell}},...,{\bf u}^{(\ell)}_{1}...{\bf u}^{(\ell)}_{r_{S,\ell}}]
\in \mathbb{R}^{n\times (r_{0,\ell} + \sum\limits_{s=1,...,S} r_{s,\ell})}, \quad \ell=1,2,3.
\end{equation}
Given the rank parameter ${\bf r}=(r_1,r_2,r_3)$, introduce the truncated SVD of 
$\widehat{U}^{(\ell)}$,
$$
\widehat{U}^{(\ell)}\approx {Z}_0^{(\ell)} D_{\ell,0} {V_0^{(\ell)}}^T,\quad
{Z}_0^{(\ell)}\in \mathbb{R}^{n\times r_\ell}, \quad 
{V_0}^{(\ell)}\in \mathbb{R}^{(r_{0,\ell} + \sum\limits_{s=1,...,S} r_{s,\ell})\times r_\ell},
$$
where  $D_{\ell,0}=\mbox{diag} \{\sigma_{\ell,1},\sigma_{\ell,2},...,\sigma_{\ell,r_\ell}\}$.
Here instead of fixed rank parameter the truncation threshold $\varepsilon >0$ can be chosen.

Now items (a) - (d) in Theorem \ref{thm:C2Tucker} 
can be generalized to the case of Tucker tensors.
In particular, the stability criteria for RHOSVD approximation as 
in Lemma \ref{lem:Can_stab} allows natural extension to the case of 
generalized RHOSVD approximation applied to a sum of Tucker tensors in
(\ref{eqn:Tuck_LatticeSum_full}).

\begin{figure}[htbp]
\centering
\includegraphics[width=7.6cm]{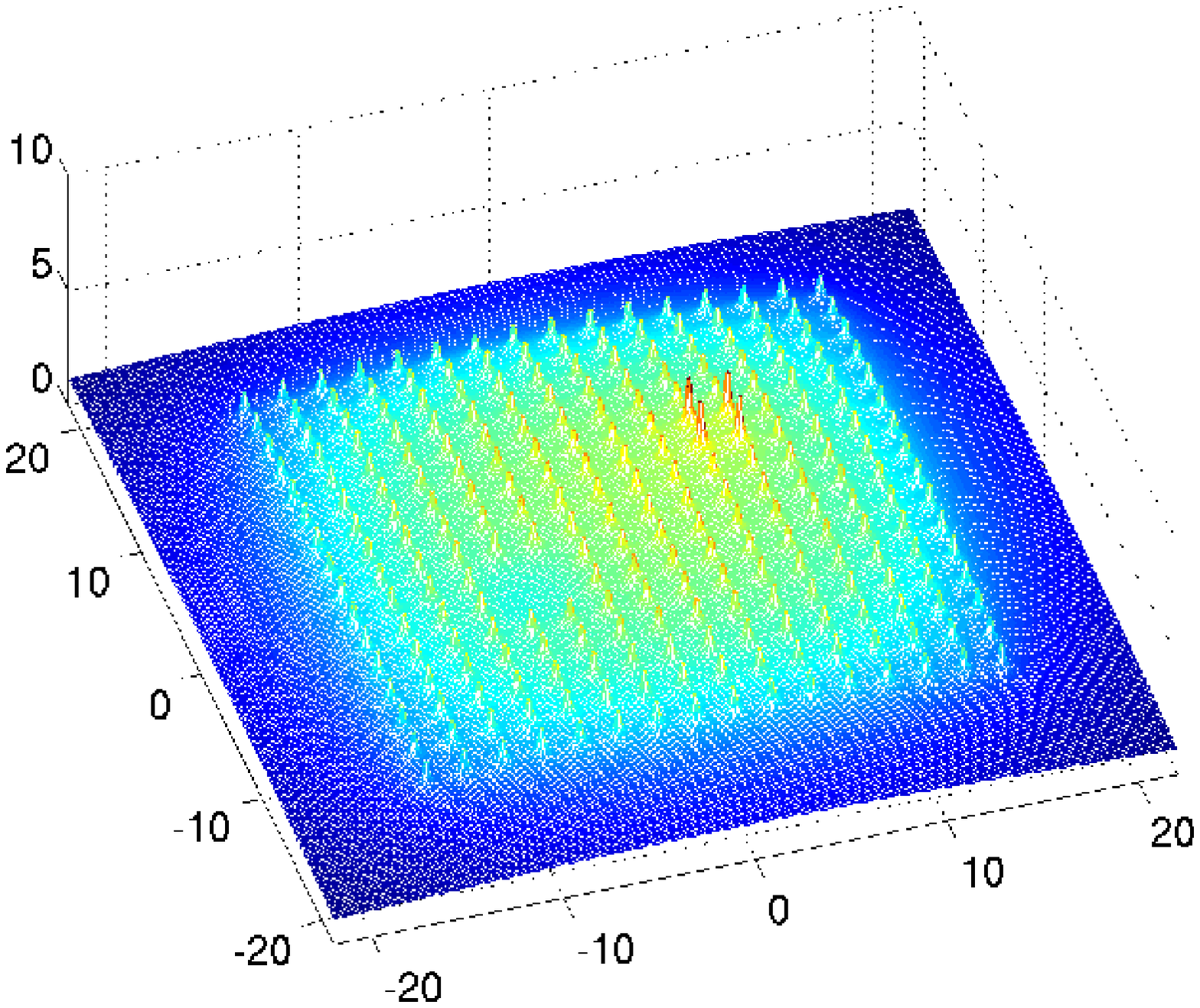}
\includegraphics[width=6.0cm]{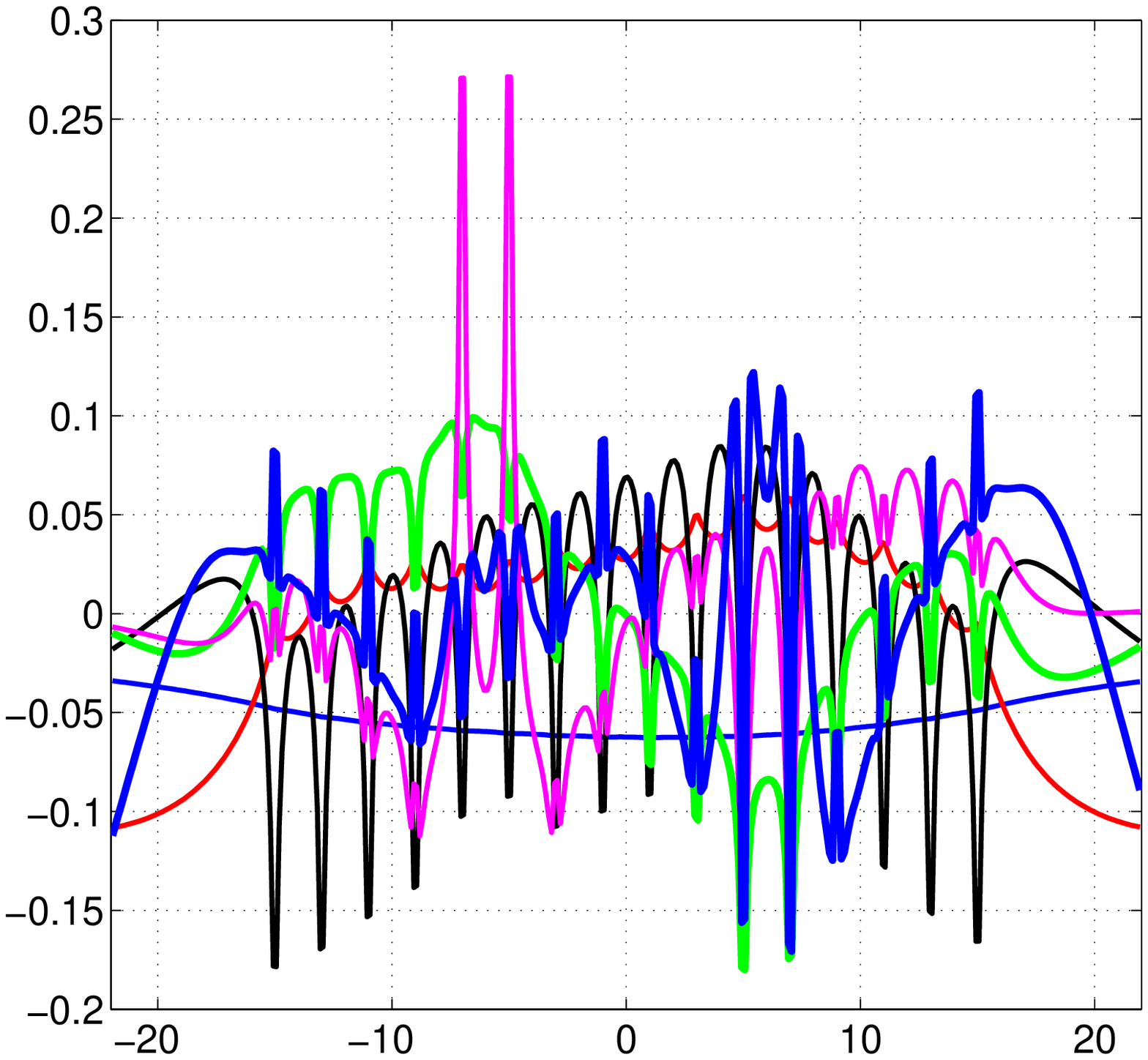}
\caption{Left: assembled grid-based Tucker sum of 3D Newton potentials on a lattice 
$16 \times 16\times 1$, with an impurity/vacancy of size $2\times 2\times1$. 
Right: the Tucker vectors along $x$-axis.}
\label{fig:Tuck_vac}  
\end{figure}

The following theorem provides the error estimate for the generalized RHOSVD approximation 
converting a sum of Tucker tensors to a single Tucker tensor with fixed rank bounds, or 
subject to the given tolerance $\varepsilon >0$.

\begin{theorem} \label{thm:TSum2Tucker} (Tucker-Sum-to-Tucker) \\
Given a sum of Tucker tensors (\ref{eqn:Tuck_LatticeSum_full}) and the rank truncation
parameter ${\bf r}=(r_1,...,r_d)$.\\
(a) Let $\sigma_{\ell,1}\geq\sigma_{\ell,2} ...  \geq\sigma_{\ell,\min(n,R)}$
 be the  singular values of the $\ell$-mode side-matrix
 $\widehat{U}^{(\ell)}\in \mathbb{R}^{n\times R}$ ($\ell=1,2,3$) defined in (\ref{eqn:SideMTuckS}).
 Then the generalized RHOSVD approximation 
 ${\bf U}_{({\bf r})}^0$ 
 obtained by the projection  of $\widehat{\bf U}$ onto the dominating singular 
vectors $Z_0^{(\ell)}$ of the Tucker side-matrices,
$\widehat{U}^{(\ell)} \approx {Z}_0^{(\ell)} D_{\ell,0} {V_0^{(\ell)}}^T $, 
exhibits  the error estimate
\begin{equation}\label{eqn:error_bound_TtoT}
\|\widehat{\bf U}- {\bf U}_{({\bf r})}^0\| \leq |\widehat{\bf B}| \sum\limits_{\ell=1}^d
(\sum\limits_{k=r_\ell +1}^{\min(n,\widehat{r_\ell})} \sigma_{\ell,k}^2)^{1/2},
\quad \mbox{where}\quad
|\widehat{\bf B}|^2 = \sum\limits_{s =0}^{S} \|{\bf B}_{s}\|^2.
\end{equation}
(b) Assume the stability condition for the sum (\ref{eqn:Tuck_LatticeSum_full}),
\[
 \sum\limits_{s =0}^{S} \|{\bf B}_{s}\|^2 \leq C \|\widehat{\bf U}\|^2,
\]
then the generalized RHOSVD approximation provides the quasi-optimal error bound
\[
 \|\widehat{\bf U}- {\bf U}_{({\bf r})}^0\| \leq C \|\widehat{\bf U}\| \sum\limits_{\ell=1}^d
(\sum\limits_{k=r_\ell +1}^{\min(n,\widehat{r_\ell})} \sigma_{\ell,k}^2)^{1/2}.
\]
\end{theorem}
\begin{proof} 
Proof of item (a) is similar to those for Theorem \ref{thm:C2Tucker}, 
presented in \cite{KhKh3:08}. Item (b) follows from (\ref{eqn:error_bound_TtoT})
taking into account the stability condition.
\end{proof}

The resultant Tucker tensor ${\bf U}_{({\bf r})}^0$ can be considered as the initial guess for the ALS iteration
to compute best Tucker $\varepsilon$-approximation of a sum of Tucker tensors.

Figure \ref{fig:Tuck_vac} (left) visualizes result of assembled Tucker summation of 3D grid-based Newton 
potentials on a $16 \times 16\times 1$ lattice, with a vacancy and impurity, each 
of $2\times 2\times1$ lattice size. 
Figure \ref{fig:Tuck_vac} (right) shows the corresponding Tucker vectors along $x$-axis.
These vectors clearly represent the local shape of vacancies and impurities. 


\subsection{Summation over non-rectangular lattices}
\label{ssec:NonRect_lattice}

In many practically interesting cases the physical lattice may have 
non-rectangular geometry that
does not fit exactly the tensor-product structure of the canonical/Tucker data arrays.
For example, the hexagonal or parallelepiped type lattices as well as their combination can be considered.
The case study of many particular classes of geometries is beyond the scope of our paper.
Instead, we formulate the main principles on how to apply tensor summation methods to
certain classes of non-rectangular geometries and give a few examples demonstrating the required (minor) 
modifications of the basic agglomerated summation schemes described above. 

\begin{figure}[t]
\begin{minipage}[t]{0.4\linewidth}
\centering
\caption{Hexagonal lattice is a union of two rectangular lattices, ''red`` and ''blue``}
\label{fig:Hexa_grid}
\vspace{0.5cm}
\begin{tikzpicture}[scale=0.5]
\foreach \x in {1.5,3.5,5.5,7.5,9.5} {
\foreach \y in {1.5,3.5,4.5,6.5,7.5,9.5} {
\shade[shading = ball, ball color = blue] (\x,\y) circle (.2);
}}
\foreach \x in {0.5,2.5,4.5,6.5,8.5,10.5}{
\foreach \y in {2,3,5,6,8,9} {
\shade[shading = ball, ball color = blue] (\x,\y) circle (.2);
}}

\foreach \y in {1.5,3.5,4.5,6.5,7.5,9.5} {
\path[draw] [color=blue](0,\y) -- (11,\y);  }

\foreach \x in {1.5,3.5,5.5,7.5,9.5} {
\path[draw] [color=blue](\x,1) -- (\x,10);  }

\foreach \y in {2,3,5,6,8,9} {
\path[draw] [color=red](0,\y) -- (11,\y);  }

\foreach \x in {0.5,2.5,4.5,6.5,8.5,10.5}{
\path[draw] [color=red](\x,1) -- (\x,10);  }
\end{tikzpicture}
\end{minipage}
\hspace{5mm}
\begin{minipage}[t]{0.49\linewidth}
\centering
\caption{Parallelogram-type lattice}
\label{fig:para_grid}

\vspace{1.5cm}
\begin{tikzpicture}[scale=0.7]
\foreach \x in {1,2,3,4,5,6,7,8} {
\shade[shading = ball, ball color = blue] (\x,1) circle (.15);
}

\foreach \x in {1,2,3,4,5,6,7,8} {
\shade[shading = ball, ball color = blue] (\x+0.2,2) circle (.15);
}
\foreach \x in {1,2,3,4,5,6,7,8} {
\shade[shading = ball, ball color = blue] (\x+0.4,3) circle (.15);
}
\foreach \x in {1,2,3,4,5,6,7,8} {
\shade[shading = ball, ball color = blue] (\x+0.6,4) circle (.15);
}

\foreach \x in {1,2,3,4,5,6,7,8} {
\shade[shading = ball, ball color = blue] (\x+0.8,5) circle (.15);
}

\foreach \y in {1,2,3,4,5} {
\path[draw] [color=red,thick](1,\y) -- (9,\y);  }
\end{tikzpicture}
\end{minipage}
\end{figure}

It is worth to note that most of interesting lattice structures 
(say, arising in crystalline modeling) inherit a number of spacial symmetries which 
allow, first, to classify and then simplify the computational schemes for each 
particular case of symmetry.
In this concern, we consider  the following classes of lattice topologies which can be 
efficiently treated by our tensor summation techniques: 
\begin{itemize}
\item[(A)] The target lattice ${\cal L}$ can be split into the union of several (few)
sub-lattices, ${\cal L}= \bigcup {\cal L}_q$, such that each sub-lattice ${\cal L}_q$
allows a 3D rectangular grid-structure.

\item[(B)] The 3D lattice points belong to the rectangular tensor grid  
in two spatial coordinates, but they violate  the tensor structure in the third variable
(say, parallelogram type grids).

\item[(C)] The 3D lattice points belong to the tensor grid in one of spatial 
coordinate, but they may violate the rectangular tensor structure in the 
remaining couple of variables.

\item[(D)] Defects in the target lattice are distributed over rectangular sub-lattices (clusters) 
represented on several coarser scales (multi-level tensor lattice sum).
\end{itemize}

\begin{figure}[htbp]
\centering
\includegraphics[width=6cm]{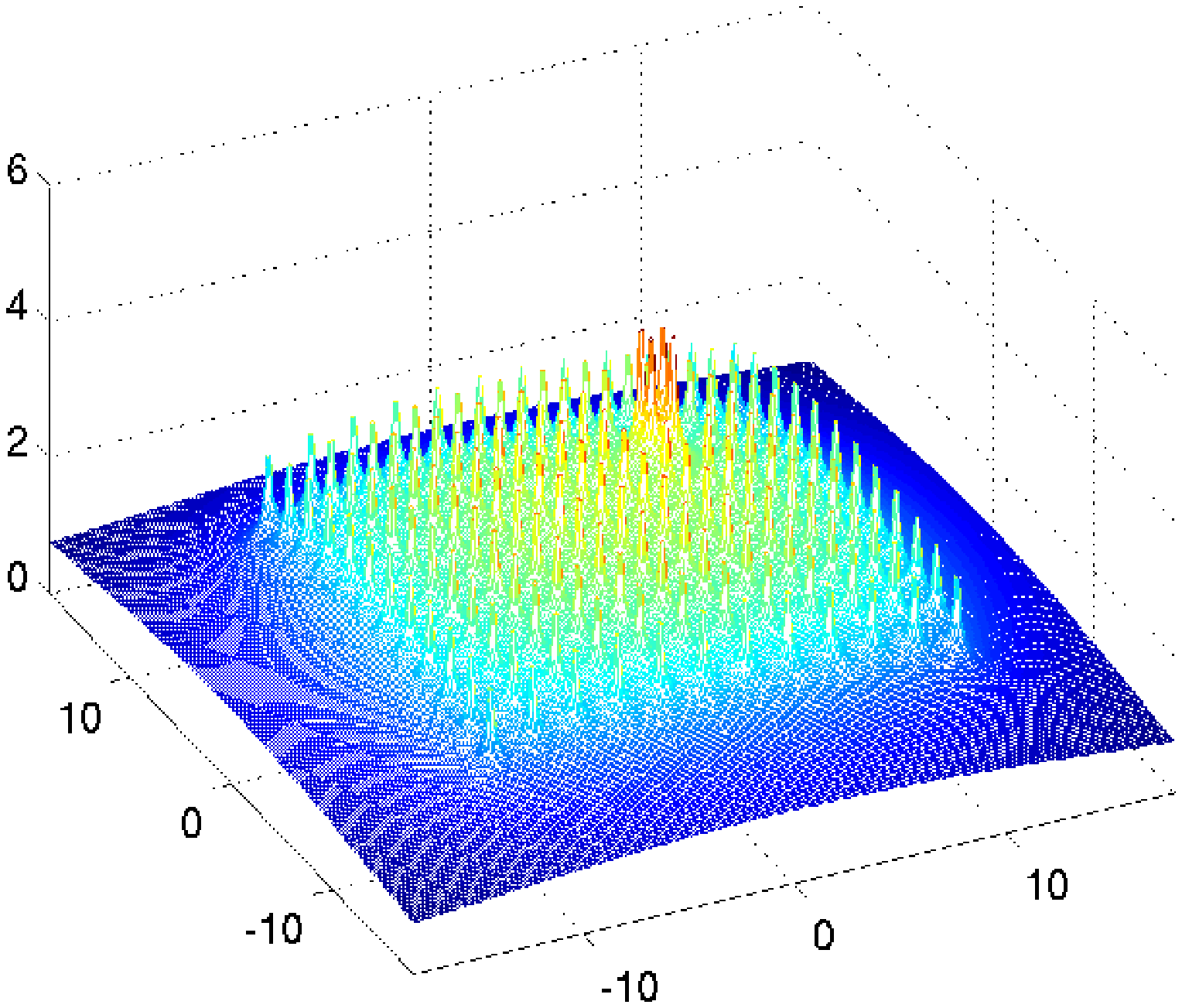}\quad
\includegraphics[width=5cm]{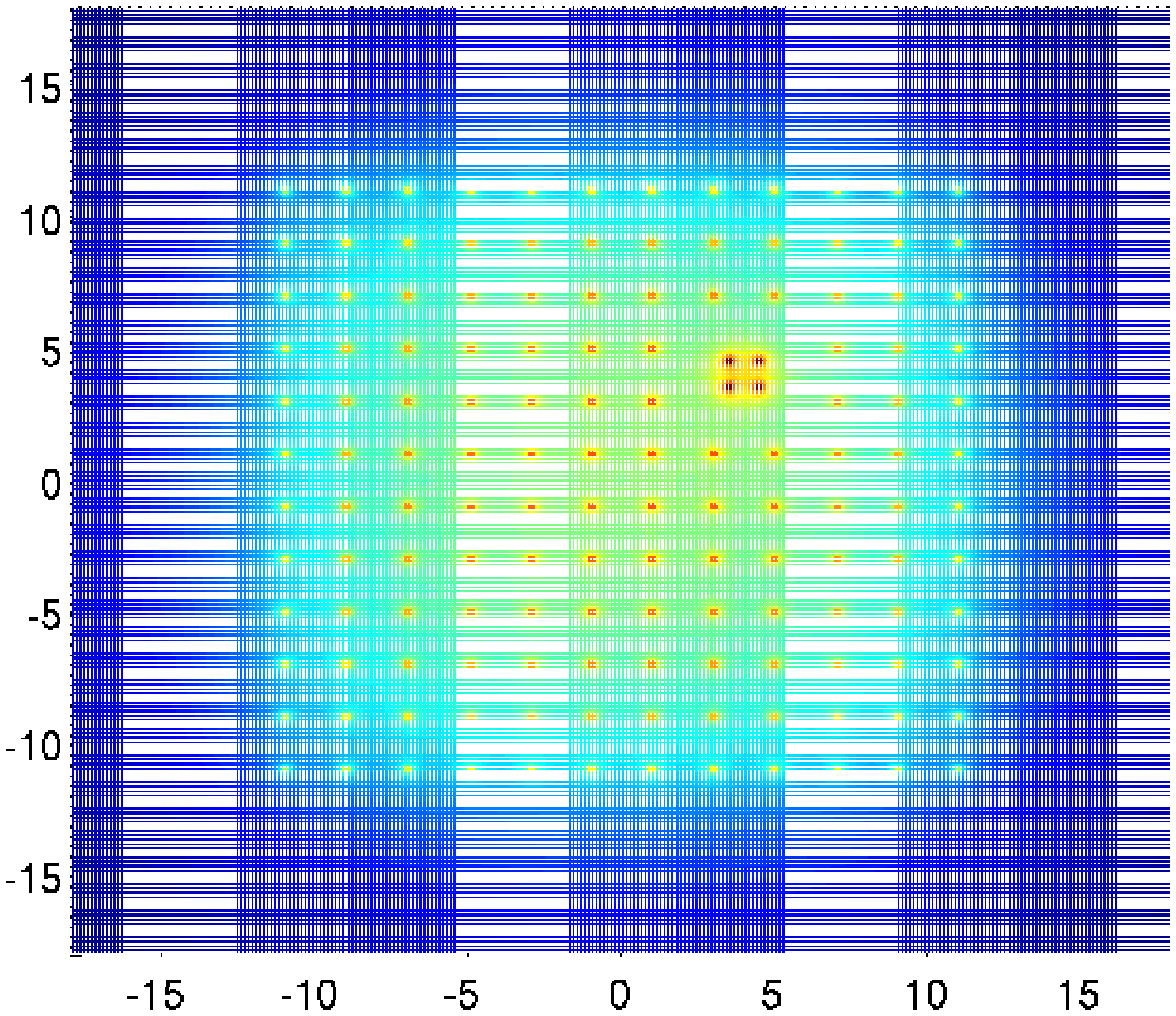}
\caption{Left: assembled canonical summation of 3D grid-based Newton potentials on a lattice 
$12 \times 12\times 1$, with an impurity, of size $2\times 2\times1$. 
Right: the vertical projection.}
\label{fig:Can_vac2_shift}  
\end{figure}

In case (A) the agglomerated tensor summation algorithms apply independently to each rectangular
sub-lattice ${\cal L}_q$, and then the target tensor is obtained as a direct 
sum of tensors associated with ${\cal L}_q$, supplemented by the subsequent
rank reduction procedure. The example of such a geometry is given by hexagonal grid presented in
Figure \ref{fig:Hexa_grid}, left ($(x,y)$ section of the $3$D lattice, that is 
rectangular in $z$-direction), 
which can be split into a union of two rectangular sub-lattices
${\cal L}_1$ (red) and ${\cal L}_2$ (blue). Another example is a lattice with $L$-shape
boundary. In this case the maximal rank does not exceed the multiple of $2$ and
the rank of a single reference Tucker tensor.

In case (B) the tensor summation applies only in two indices while a sum in the remaining 
third index is treated directly. This leads to the increase of directional rank 
proportionally to the 1D size of the lattice, $L$, 
hence requiring the subsequent rank reduction procedures described in 
\S\ref{ssec:Can2Tuck_rev} and \S\ref{ssec:TuckSum2Tuck}.
This may lead to the higher computational complexity of the summation.
An example of such a structure is the parallelogram-type lattice 
shown in Figure \ref{fig:para_grid}, right (orthogonal projection onto $(x,y)$ plane).

In case (C) the agglomerated summation can be performed only
in one index, supplemented by the direct summation in the remaining indices. 
The total rank then increases proportionally to $L^2$, 
making the subsequent rank optimization procedure indispensable. 
However, even in this worst case scenario the asymptotic complexity of the direct summation
shall be reduced on the order of magnitude in $L$ from $O(L^3)$ to  $O(L^2)$
due to the benefits of ''one-way'' tensor summation.

Case (D) can be treated by successive application of the canonical/Tucker tensor 
summation algorithm at several levels of defects location.
Figure \ref{fig:Can_vac2_shift} represent the result of assembled canonical summation 
of 3D grid-based Newton potentials on a lattice 
$12 \times 12\times 1$, with an impurity of size $2\times 2\times1$ that does not fit the location
of lattice points.  Since the impurity potentials are determined on the same fine 
$N_L \times N_L \times N_L$ representation grid, the difference  in inter-potential distances
does not influence on the numerical treatment of the defects. 
In the case of many non-regularly distributed defects the summation should be implemented 
in the Tucker format with the subsequent rank truncation. 

Figure \ref{fig:Vacance_L_O} (left) visualizes the result of assembled canonical summation of 3D 
grid-based Newton potentials on a lattice $24 \times 24\times 1$, with  regularly positioned  
$ 6\times 6\times 1$ vacancies (two-level lattice). Figure \ref{fig:Vacance_L_O} represents the result of
assembled canonical summation of the Newton potentials on $L$-shaped (left) 
and $O$-shaped (right) sub-lattices of the $24 \times 24\times 1$ lattice (two-level step-type geometry).
In all these cases the total tensor rank does not exceed the double rank of the single reference potential
since all vacancies are located on tensor sub-lattice of the target lattice.

\begin{figure}[htbp]
\centering
\includegraphics[width=4.6cm]{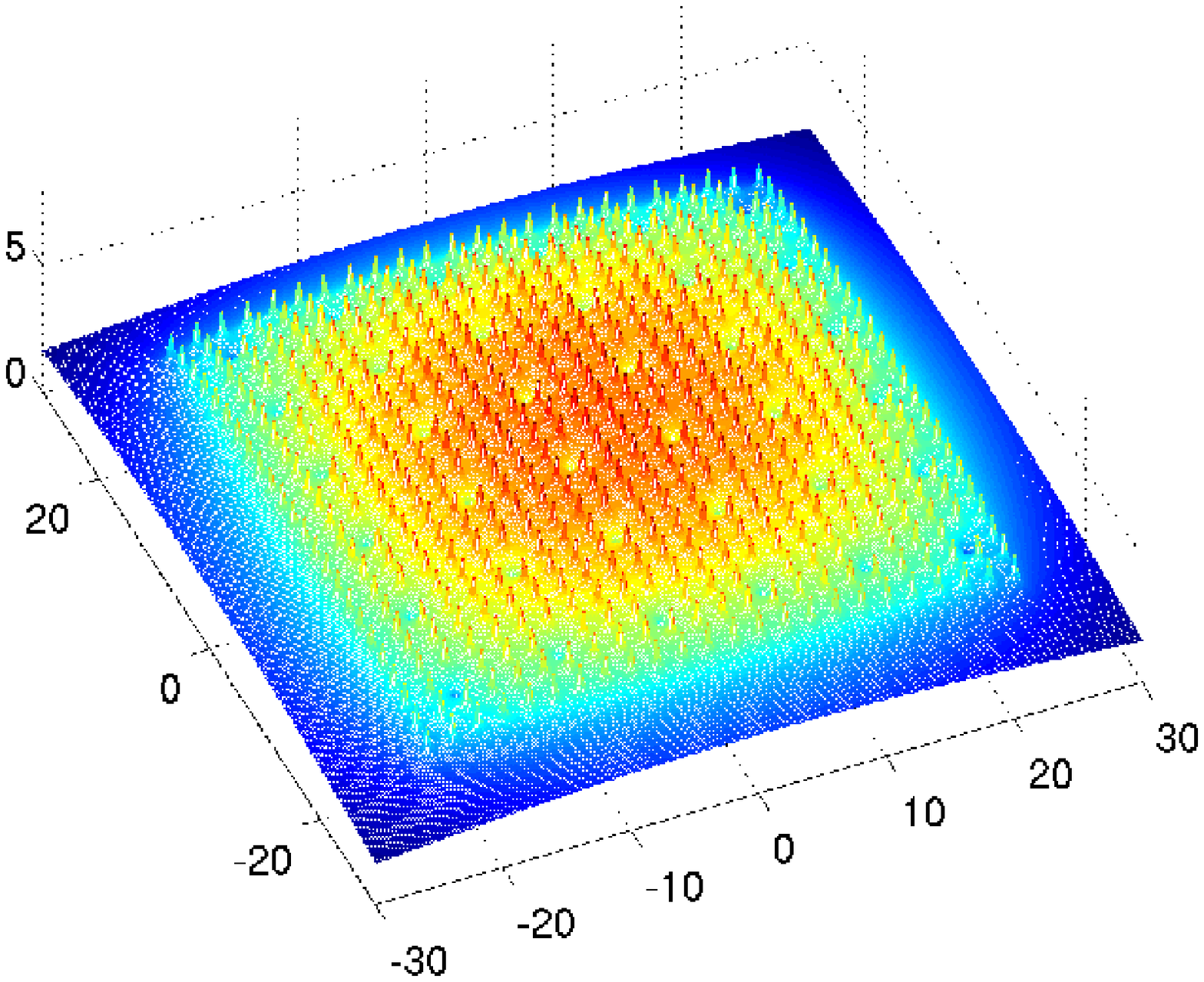}\quad
\includegraphics[width=4.6cm]{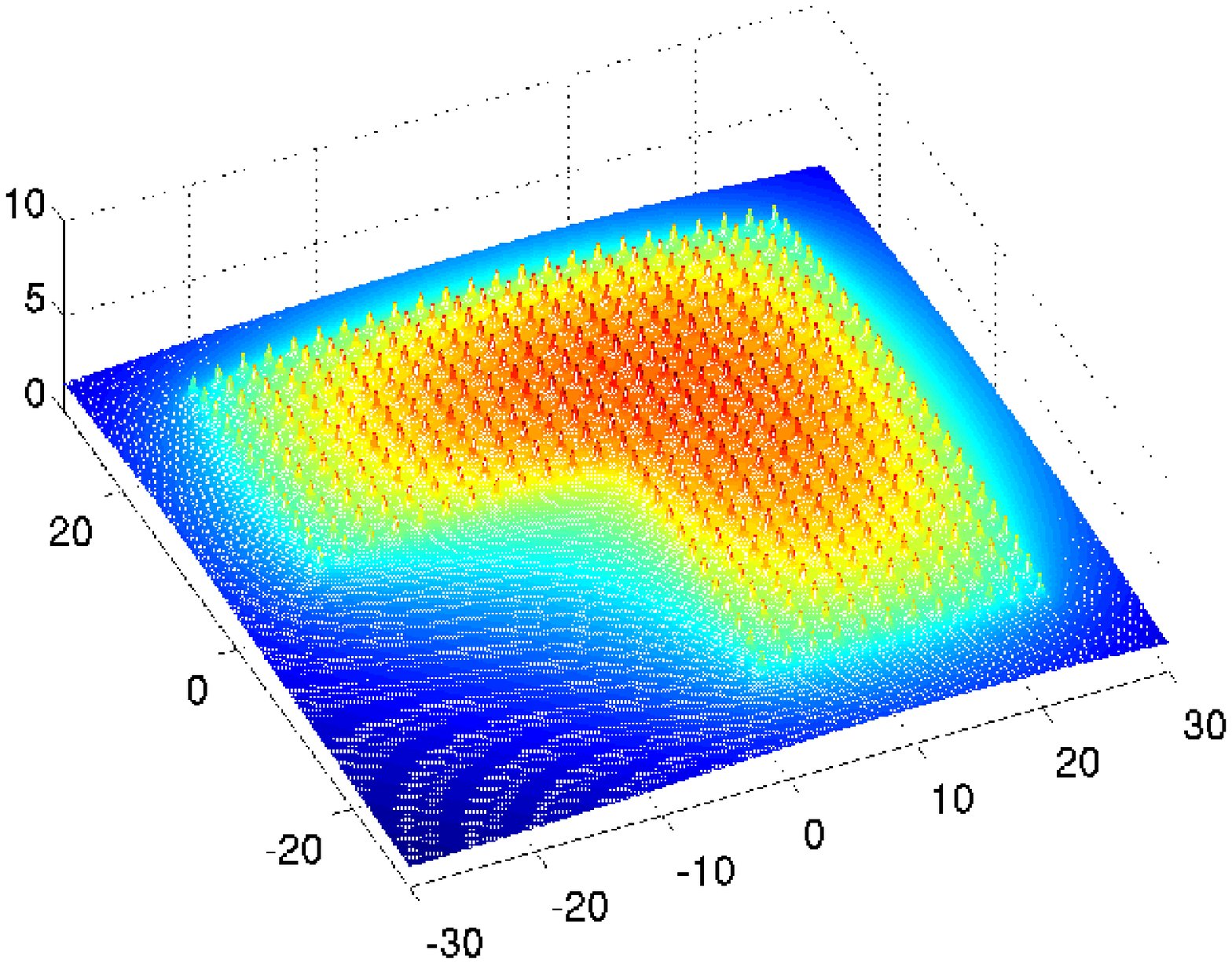}\quad
\includegraphics[width=4.6cm]{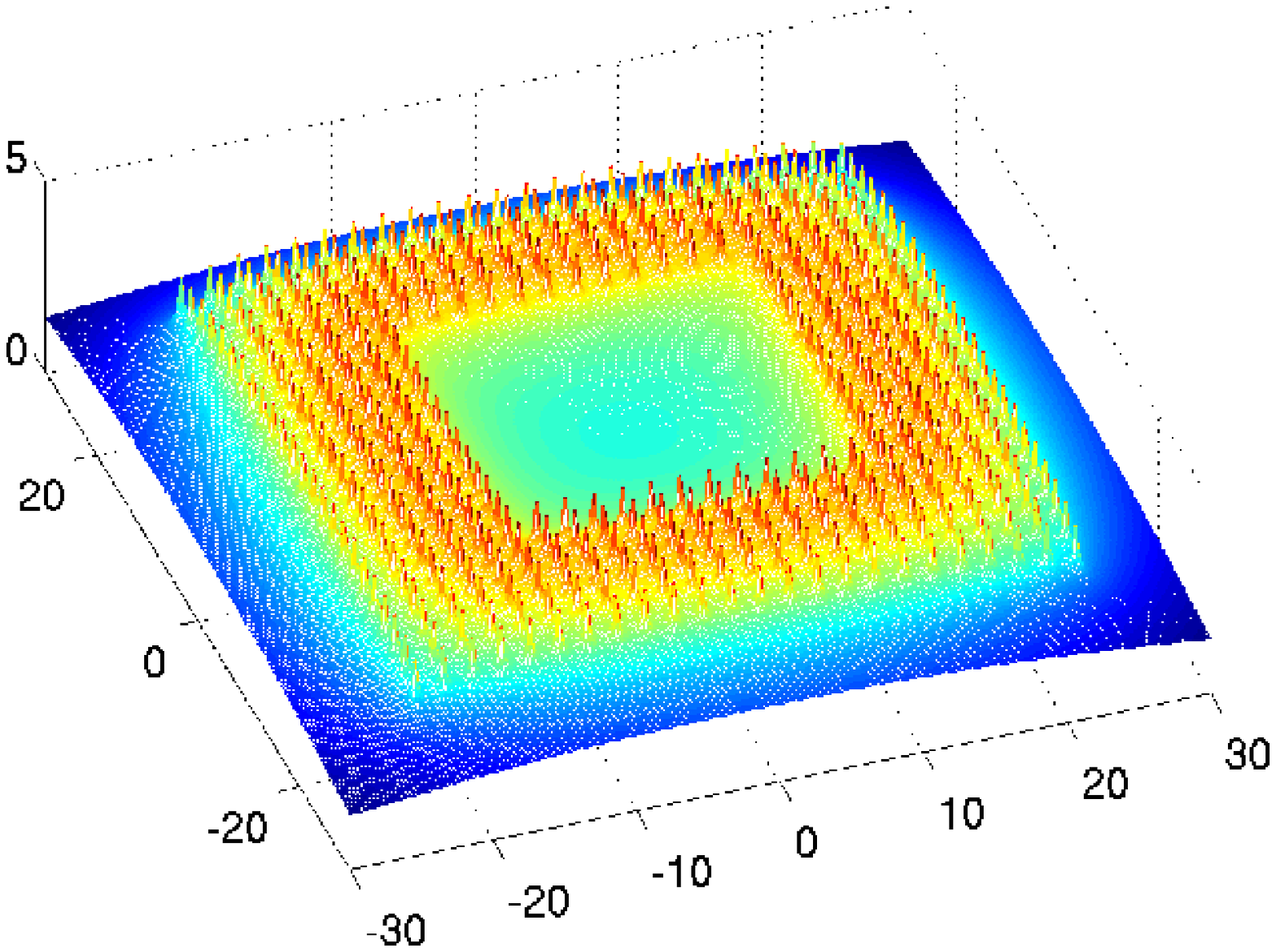}
\caption{Assembled summation of 3D grid-based Newton potentials in canonical format 
on a $24 \times 24\times 1$ lattice: (left)  regular $ 6\times 6\times 1$ vacancies, 
(middle) $L$-shaped geometry, (right) $O$-shaped sub-lattices.}
\label{fig:Vacance_L_O}  
\end{figure}

We summarize  that in all cases (A) - (D) classified  above
the tensor summation approach cab be gainfully applied.
The overall numerical cost may depend on the geometric structure 
and symmetries of the system under consideration since violation of the tensor-product 
rectangular structure of the lattice may lead to the increase in the Tucker/canonical rank.
This is clearly observed in the case of moderate number of defects distributed randomly. 
In all such cases the RHOSVD approximation combined with 
the ALS iteration serves for the robust rank reduction in the Tucker format.

\section{Conclusions}
\label{sec:Conclusion}

In this paper we presented the fast rank-structured tensor  method for the efficient grid-based
summation of long-range potentials on lattices with vacancies and defects, as well as
in the presence of non-rectangular geometries. It is shown that summation of potentials on
perturbed $L\times L \times L$  lattices by using the Tucker/canonical  tensor formats
can be performed in $O(L)$ operations that improves dramatically the cost $O(L^3)$ by the standard methods.

All computational 3D data are presented on the one common fine $N\times N \times N$ grid
by low-rank tensors in $\mathbb{R}^{N \times N \times N}$, that
allows the simultaneous approximation with guaranteed precision of all singular kernel
functions involved in the summation.
In case of unperturbed lattice, both the canonical and Tucker ranks of the resultant tensor sum
remains the same as for the individual reference potential.

Calculation of the potential sum on defected lattices  is performed in an algebraic
way, by using summation rules for tensors in the canonical or Tucker formats, which lead to
increase in the Tucker or canonical ranks of the resultant tensor. The rank truncation
for  the overall potential sum is based on the canonical-to-Tucker
or Tucker-sum-to-Tucker transform via the reduced HOSVD approximation.
The stability conditions for such kind of approximation have been analyzed. 

The presented approach yields enormous reduction in storage and computing time.
Numerical examples illustrate the rank bounds and asymptotic complexity of the
tensor summation method in both canonical and Tucker  data formats in the agreement
with theoretical predictions. Summation of millions of potentials on a finite 3D lattice
is performed in seconds in Matlab implementation.

This scheme can be applied to a number of potentials including the Newton, Slater, Yukawa,
Lennard-Jones, Buckingham and dipole-dipole kernel functions.
The assembled tensor summation approach is well suited for further applications in electronic
and molecular structure calculations of  large lattice-structured compounds,
see \cite{VeKhorCorePeriod:14}, as well as in various computational problems for many-particle systems. 
In particular, it is can be efficient for calculation of electronic properties of large finite 
crystalline systems like quantum dots, which are 
intermediate between bulk (periodic) systems and discrete molecules.

\section{Appendix: Canonical-to-Tucker approximation}
\label{sec:Appendix}

In Appendix we present the error estimate for the RHOSVD approximation by the
so-called Canonical-to-Tucker scheme \cite{KhKh3:08}.
Let us denote by ${\cal G}_{\ell}$ the so-called Grassman manifold
that is a factor space with respect to all possible rotations 
to the Stiefel manifold $ {\cal M}_{\ell} $ of orthogonal $n \times r_\ell$ matrices,
\begin{equation*}\label{Stiefel m}
{\cal M}_{\ell}:=
\{Y\in \mathbb{R}^{n \times r_\ell}:Y^TY=I_{r_\ell\times r_\ell} \},   
 \quad (\ell=1,...,d).
\end{equation*}
Denote by ${\cal T}_{{\bf r},{\bf n}}$ the set of rank-${\bf r}$ Tucker tensors.

\begin{theorem} \label{thm:C2Tucker}
(Canonical to Tucker approximation, \cite{KhKh3:08}). \\
(a) Let ${\bf A}= {\bf A}_{(R)}$ 
be given by (\ref{eqn:CP_form}). Then the minimization problem
\begin{equation}\label{Cost f1}
  {\bf A}\in \mathbb{V}_{\bf n}:
  \quad {\bf A}_{({\bf r})} =
  \operatorname{argmin}_{{\bf T} \in {\cal T}_{{\bf r},{\bf n}}}
 \|{\bf A} - {\bf T}\|_{\mathbb{V}_{\bf n}},
\end{equation}
is equivalent to the {\it dual maximization problem} over  the Grassman 
manifolds $ {\cal G}_{\ell} $,
\begin{equation}\label{GRQ Max CR2T}
  [{ W}^{(1)},...,{ W}^{(d)}] = 
\operatorname{argmax}_{{Y}^{(\ell)}\in {\cal G}_{\ell}  }
  \left\| \sum\limits_{\nu =1}^{R} \xi_{ \nu}
    \left({{ Y}^{(1)}}^T\, {\bf a}^{(1)}_\nu \right) \otimes ... 
\otimes \left( {{ Y}^{(d)}}^T\, {\bf a}^{(d)}_\nu \right)
    \right\|^2_{\mathbb{R}^{\bf r}}, 
\end{equation}
where ${Y}^{(\ell)}=[{y}^{(\ell)}_{1}...{y}^{(\ell)}_{r_\ell}]
\in \mathbb{R}^{n\times r_\ell} $ ($\ell=1,...,d$), and
$ {{Y}^{(\ell)}}^T\, {\bf a}^{(\ell)}_\nu \in \mathbb{R}^{r_\ell} $.

(b)
The compatibility condition 
$
r_\ell \leq rank(A^{(\ell)})\; \mbox{with}\; 
A^{(\ell)}=[{\bf a}^{(\ell)}_1 ... {\bf a}^{(\ell)}_R ]\in \mathbb{R}^{n\times R}
$
being the $\ell$-mode side-matrix, ensures the solvability of (\ref{GRQ Max CR2T}).
The maximizer is given by orthogonal
matrices ${ W}^{(\ell)}=[{\bf w}^{(\ell)}_{1} ... {\bf w}^{(\ell)}_{r_\ell}]\in
\mathbb{R}^{n\times r_\ell}  $, which  can be computed by
ALS Algorithm with the initial guess chosen as
the reduced HOSVD approximation of ${\bf A}$ given by ${\bf A}_{({\bf r})}^0$, see Definition \ref{def:RHOSVD}.


(c) Precomputed matrices ${ W}^{(\ell)}$,
 the minimizer in (\ref{Cost f1}) is then calculated by the orthogonal projection
\[
{\bf A}_{({\bf r})}=\sum\limits_{{\bf k}={\bf 1}}^{\bf r}  \mu_{\bf k}
   {\bf w}^{(1)}_{k_1}\otimes \cdots \otimes {\bf w}^{(d)}_{k_d}, 
 \quad \mu_{\bf k}= \langle {\bf w}^{(1)}_{k_1}\otimes \cdots \otimes {\bf w}^{(d)}_{k_d}
 , {\bf A} \rangle,
\]
where the core tensor $\boldsymbol{\mu} =[\mu_{\bf k} ]$ can be represented
in the rank-$R$ canonical format 
\begin{equation*}\label{C-core}
\boldsymbol{\mu}=\sum\limits_{\nu =1}^{R} \xi_{ \nu}({{ W}^{(1)}}^T\,
{\bf a}^{(1)}_\nu) \otimes \cdots \otimes ({{W}^{(d)}}^T\, {\bf a}^{(d)}_\nu).
\end{equation*}

(d) Let 
$\sigma_{\ell,1}\geq\sigma_{\ell,2} ...  \geq\sigma_{\ell,\min(n,R)}$
 be the  singular values of the $\ell$-mode side-matrix
 $A^{(\ell)}\in \mathbb{R}^{n\times R}$ ($\ell=1,...,d$).
 Then the reduced HOSVD approximation  ${\bf A}_{({\bf r})}^0$ 
exhibits  the error estimate
\begin{equation}\label{eqn:error_bound_CtoT}
\|{\bf A}- {\bf A}_{({\bf r})}^0\| \leq \|\boldsymbol{\xi}\| \sum\limits_{\ell=1}^d
(\sum\limits_{k=r_\ell +1}^{\min(n,R)} \sigma_{\ell,k}^2)^{1/2},\quad \mbox{where}\quad
\|\boldsymbol{\xi}\|^2 = \sum\limits_{\nu =1}^{R} \xi_{\nu}^2.
\end{equation}
\end{theorem}

\begin{footnotesize}

\end{footnotesize}

\end{document}